\documentclass[reqno]{amsart}
\usepackage[latin1]{inputenc}
\usepackage[pdftex]{graphicx}
\usepackage{graphics}
\usepackage{graphicx}
\usepackage{epstopdf}
\usepackage[pdftex,hyperindex]{hyperref}
\usepackage{floatflt}
\usepackage{multicol}
\usepackage[T1]{fontenc}
\usepackage{textcomp}
\usepackage{latexsym}
\usepackage{expdlist}
\usepackage{vmargin}
\usepackage{booktabs}
\usepackage{placeins}
\usepackage{array}
\usepackage{amsmath}
\usepackage{amssymb}
\usepackage{amsfonts}
\usepackage{verbatim}
\usepackage{array}
\usepackage{framed}
\usepackage{amsthm}
\usepackage{enumerate}
\usepackage{cite}
\input amssym.def
\input amssym
\theoremstyle{plain}
\newtheorem{theorem}{Theorem}[section]
\newtheorem{lemma}{Lemma}[section]
\newtheorem{corollary}{Corollary}[section]
\newtheorem{proposition}{Proposition}[section]
\theoremstyle{remark}
\newtheorem{remark}{Remark}[section]

\theoremstyle{definition}

\newtheorem{example}{Example}[section]

\newcommand\id{\mathrm{id}}

\allowdisplaybreaks[1]

\begin{document}

\title{Dual spaces vs. Haar measures of polynomial hypergroups}

\author{Stefan Kahler}
\address{Stefan Kahler, Fachgruppe Mathematik, RWTH Aachen University, Pontdriesch 14-16, 52062 Aachen, Germany; Lehrstuhl A f\"{u}r Mathematik, RWTH Aachen University, 52056 Aachen, Germany}
\email{kahler@mathematik.rwth-aachen.de}

\author{Ryszard Szwarc}
\address{Ryszard Szwarc, Institute of Mathematics, University of Wroc\l{}aw, pl. Grunwaldzki 2/4, 50-384 Wroc\l{}aw, Poland}
\email{szwarc2@gmail.com}

\thanks{The joint work was started while the first author (employed at RWTH Aachen University) was visiting the Institute of Mathematics of the University of Wroc\l{}aw, whose hospitality is greatly acknowledged. The authors thank Rupert Lasser who had a look at the manuscript and gave the feedback that he is not aware of further previous results which should be cited. The graphics were made with Maple.}

\date{\today}

\begin{abstract}
Many symmetric orthogonal polynomials $(P_n(x))_{n\in\mathbb{N}_0}$ induce a hypergroup structure on $\mathbb{N}_0$. The Haar measure is the counting measure weighted with $h(n):=1/\int_\mathbb{R}\!P_n^2(x)\,\mathrm{d}\mu(x)\geq1$, where $\mu$ denotes the orthogonalization measure. We observed that many naturally occurring examples satisfy the remarkable property $h(n)\geq2\;(n\in\mathbb{N})$. We give sufficient criteria and particularly show that $h(n)\geq2\;(n\in\mathbb{N})$ if the (Hermitian) dual space $\widehat{\mathbb{N}_0}$ equals the full interval $[-1,1]$, which is fulfilled by an abundance of examples. We also study the role of nonnegative linearization of products (and of the harmonic and functional analysis resulting from such expansions). Moreover, we construct two example types with $h(1)<2$. To our knowledge, these are the first such examples. The first type is based on Karlin--McGregor polynomials, and $\widehat{\mathbb{N}_0}$ consists of two intervals and can be chosen ``maximal'' in some sense; $h$ is of quadratic growth. The second type relies on certain compact operators; $h$ grows exponentially, and $\widehat{\mathbb{N}_0}$ is discrete.
\end{abstract}



\maketitle

\numberwithin{equation}{section}

\section{Introduction}\label{sec:intro}

\subsection{Basic setting and observation}

Let $(P_n(x))_{n\in\mathbb{N}_0}\subseteq\mathbb{R}[x]$ with $\mathrm{deg}\;P_n(x)=n$ be given by some recurrence relation $P_0(x)=1$, $P_1(x)=x$,
\begin{equation}\label{eq:threetermrec}
x P_n(x)=a_n P_{n+1}(x)+c_n P_{n-1}(x)\;(n\in\mathbb{N}),
\end{equation}
where $(c_n)_{n\in\mathbb{N}}\subseteq(0,1)$ and $a_n\equiv1-c_n$; to avoid case differentiations, we additionally define $a_0:=1$. Obviously, the resulting polynomials are symmetric and normalized by $P_n(1)\equiv1$. It is well-known from the theory of orthogonal polynomials\footnote{Standard results on orthogonal polynomials can be found in \cite{Ch78}, for instance.} that $(P_n(x))_{n\in\mathbb{N}_0}$ is orthogonal w.r.t. a unique probability (Borel) measure $\mu$ on $\mathbb{R}$ which satisfies $|\mathrm{supp}\;\mu|=\infty$ and $\mathrm{supp}\;\mu\subseteq[-1,1]$ (Favard's theorem), which means
\begin{equation*}
\int_\mathbb{R}\!P_m(x)P_n(x)\,\mathrm{d}\mu(x)\neq0\Leftrightarrow m=n.
\end{equation*}
Moreover, it is well-known that the zeros of the polynomials are real, simple and located in the interior of the convex hull of $\mathrm{supp}\;\mu$. In particular, all $P_n$ are strictly positive at the right end point of $\mathrm{supp}\;\mu$. We are interested in sequences which satisfy the additional `nonnegative linearization of products' property
\begin{equation}\label{eq:productlinear}
P_m(x)P_n(x)=\sum_{k=0}^{m+n}\underbrace{g(m,n;k)}_{\overset{!}{\geq}0}P_k(x)\;(m,n\in\mathbb{N}_0),
\end{equation}
i.e., the product of any two polynomials $P_m(x),P_n(x)$ is a convex combination w.r.t. the basis $\{P_k(x):k\in\mathbb{N}_0\}$. Due to orthogonality, one has $g(m,n;|m-n|),g(m,n;m+n)\neq0$ and $g(m,n;k)=0$ for $k<|m-n|$, so the summation in \eqref{eq:productlinear} starts with $k=|m-n|$ (and \eqref{eq:productlinear} can be regarded as an extension of the recurrence \eqref{eq:threetermrec}) \cite{La05}. The nonnegativity of the linearization coefficients $g(m,n;k)$ gives rise to a commutative discrete hypergroup on $\mathbb{N}_0$, where the convolution $(m,n)\mapsto\sum_{k=|m-n|}^{m+n}g(m,n;k)\delta_k$ maps $\mathbb{N}_0\times\mathbb{N}_0$ into the convex hull of the Dirac functions on $\mathbb{N}_0$, the identity on $\mathbb{N}_0$ serves as involution and $0$ is the unit element.\footnote{The full hypergroup axioms can be found in standard literature like \cite{BH95}. The axioms for the special case of a discrete hypergroup are considerably simpler and can be found in \cite{La05}. Roughly speaking, hypergroups generalize locally compact groups by allowing the convolution of two Dirac measures to be no Dirac measure but a more general probability measure (satisfying certain non-degeneracy and compatibility properties).} Such hypergroups are called polynomial hypergroups, were introduced by Lasser in the 1980s and are generally very different from groups or semigroups \cite{La83}. There is an abundance of examples, and the individual behavior strongly depends on the underlying polynomials $(P_n(x))_{n\in\mathbb{N}_0}$. However, many concepts of harmonic analysis and functional analysis take a rather unified and concrete form. This makes these objects located at a fruitful and vivid crossing point between the theory of orthogonal polynomials and special functions, on the one hand, and functional and harmonic analysis, on the other hand. Recent publications deal with polynomial hypergroups vs. moment functions \cite{FGS22}, vs. amenability properties \cite{Ka21b,Ka23a}, vs. Ramsey theory \cite{KRS20} and vs. Gibbs states on graphs \cite{Vo23}. We briefly recall some basics \cite{La83,La05}. The nonnegativity of the $g(m,n;k)$ implies that
\begin{equation}\label{eq:consequencegelfand}
\{\pm1\}\cup\mathrm{supp}\;\mu\subseteq\widehat{\mathbb{N}_0}\subseteq[-1,1],
\end{equation}
where the compact set $\widehat{\mathbb{N}_0}$ is defined by
\begin{equation}\label{eq:dualdef}
\widehat{\mathbb{N}_0}:=\left\{x\in\mathbb{R}:\max_{n\in\mathbb{N}_0}|P_n(x)|=1\right\}.
\end{equation}
If $f:\mathbb{N}_0\rightarrow\mathbb{C}$ is an arbitrary function, then, for every $n\in\mathbb{N}_0$, the translation $T_n f:\mathbb{N}_0\rightarrow\mathbb{C}$ is given by
\begin{equation*}
T_n f(m)=\sum_{k=|m-n|}^{m+n}g(m,n;k)f(k);
\end{equation*}
the translation operator $T_n:\mathbb{C}^{\mathbb{N}_0}\rightarrow\mathbb{C}^{\mathbb{N}_0}$ is defined by $f\mapsto T_n f$. The corresponding Haar measure, normalized such that $\{0\}$ is mapped to $1$, is the counting measure on $\mathbb{N}_0$ weighted by the values of the Haar function $h:\mathbb{N}_0\rightarrow[1,\infty)$,
\begin{equation}\label{eq:hdef}
h(n):=\frac{1}{g(n,n;0)}=\frac{1}{\int_\mathbb{R}\!P_n^2(x)\,\mathrm{d}\mu(x)}.
\end{equation}
The orthonormal polynomials (with positive leading coefficients) $(p_n(x))_{n\in\mathbb{N}_0}$ which correspond to $(P_n(x))_{n\in\mathbb{N}_0}$ satisfy $p_n(x)=\sqrt{h(n)}P_n(x)\;(n\in\mathbb{N}_0)$ and are given by the recurrence relation $p_0(x)=1$, $p_1(x)=x/\sqrt{c_1}$,
\begin{equation*}
x p_n(x)=\alpha_{n+1}p_{n+1}(x)+\alpha_n p_{n-1}(x)\;(n\in\mathbb{N}),
\end{equation*}
where $\alpha_1=\sqrt{c_1}$ and
\begin{equation*}
\alpha_n=\sqrt{c_n a_{n-1}}=\sqrt{c_n(1-c_{n-1})}
\end{equation*}
for $n\geq2$. Moreover, the corresponding monic polynomials $(\sigma_n(x))_{n\in\mathbb{N}_0}$ fulfill $\sigma_0(x)=1$, $\sigma_1(x)=x$ and
\begin{equation}\label{eq:threetermrecmonic}
x\sigma_n(x)=\sigma_{n+1}(x)+\alpha_n^2\sigma_{n-1}(x)\;(n\in\mathbb{N}).
\end{equation}
If $f\in\ell^1(h):=\{f:\mathbb{N}_0\rightarrow\mathbb{C}:\left\|f\right\|_1<\infty\}$, $\left\|f\right\|_1:=\sum_{k=0}^\infty|f(k)|h(k)$, then $T_n f\in\ell^1(h)$ and
\begin{equation*}
\sum_{k=0}^{\infty}T_n f(k)h(k)=\sum_{k=0}^{\infty}f(k)h(k)
\end{equation*}
for every $n\in\mathbb{N}_0$. The norm $\left\|.\right\|_1$, the convolution $(f,g)\mapsto f\ast g$, $f\ast g(n):=\sum_{k=0}^\infty T_n f(k)g(k)h(k)$ and complex conjugation make $\ell^1(h)$ a semisimple commutative Banach $\ast$-algebra with unit $\delta_0$, so the polynomials $(P_n(x))_{n\in\mathbb{N}_0}$ can be studied via methods coming from Gelfand's theory. In particular, the important property \eqref{eq:consequencegelfand} is a consequence of functional analysis like Gelfand's theory. Polynomial hypergroups are accompanied by a sophisticated harmonic analysis and Fourier analysis. The orthogonalization measure $\mu$ serves as Plancherel measure, and $\widehat{\mathbb{N}_0}$ has an important interpretation as a dual object: Via the homeomorphism $\widehat{\mathbb{N}_0}\rightarrow\Delta_s(\ell^1(h))$, $x\mapsto\varphi_x$ with $\varphi_x(f):=\sum_{k=0}^\infty f(k)P_k(x)h(k)\;(f\in\ell^1(h))$, $\widehat{\mathbb{N}_0}$ can be identified with the Hermitian structure space $\Delta_s(\ell^1(h))$; the occurring element of $\ell^{\infty}$ which maps $k$ to $P_k(x)$ is called the symmetric character associated with $x\in\widehat{\mathbb{N}_0}$. If $h$ is of subexponential growth (i.e., for all $\epsilon>0$ there is some $M>0$ such that $h(n)\leq M(1+\epsilon)^n$ for all $n\in\mathbb{N}_0$), then $\mathrm{supp}\;\mu$ and the dual space $\widehat{\mathbb{N}_0}$ coincide \cite{Hu70,Sz95,Vo87,Vo88}. It is obvious from \eqref{eq:threetermrec} that $h$ is also given by
\begin{equation}\label{eq:hrec}
h(0)=1,\;h(n)=\prod_{k=1}^n\frac{a_{k-1}}{c_k}\;(n\in\mathbb{N}).
\end{equation}

Since $g(n,n;0)$ and $g(n,n;2n)$ are nonzero and $\sum_{k=0}^{2n}g(n,n;k)=1$, nonnegative linearization of products always implies that $h(n)=1/g(n,n;0)>1$ for all $n\in\mathbb{N}$. Studying various examples, we observed that all of them satisfied the stronger property $h(n)\geq2\;(n\in\mathbb{N})$. The paper is devoted to questions concerning this eye-catching observation, as well as to corresponding criteria, to the role of the dual space $\widehat{\mathbb{N}_0}$, to the role of nonnegative linearization of products (and of the harmonic/functional analysis resulting from such expansions) and to (counter) examples.

\subsection{Motivation and outline of the paper}

To start with, we give an additional and more detailed motivation for the problem: since the linearization coefficients $g(m,n;k)$ are often not explicitly known or of cumbersome structure, it may be very hard to check whether a concrete sequence $(P_n(x))_{n\in\mathbb{N}_0}$ satisfies the crucial nonnegative linearization of products property. We are not aware of any convenient characterization (in terms of the recurrence coefficients $(a_n)_{n\in\mathbb{N}}$ and $(c_n)_{n\in\mathbb{N}}$, in terms of the orthogonalization measure $\mu$ etc.). However, there are several sufficient criteria, starting with results of Askey \cite{As70} and continued by Szwarc et al. in a series of papers. One of these criteria \cite[Theorem 1 p. 966]{Sz92b} reads as follows:

\begin{theorem}\label{thm:szwarcnonnegII}
If $(c_n)_{n\in\mathbb{N}}$ is bounded from above by $1/2$ and both $(c_{2n-1})_{n\in\mathbb{N}}$ and $(c_{2n})_{n\in\mathbb{N}}$ are nondecreasing, then nonnegative linearization of products is satisfied.
\end{theorem}

Now if $(c_n)_{n\in\mathbb{N}}$ is bounded from above by $1/2$ (and thus $(a_n)_{n\in\mathbb{N}}$ is bounded from below by $1/2$) like in Theorem~\ref{thm:szwarcnonnegII}, then it is clear from \eqref{eq:hrec} that indeed $h(n)\geq2$ for all $n\in\mathbb{N}$ (recall that $a_0=1$). Therefore, it is at least not surprising that many examples satisfy this property because many examples are either constructed via Theorem~\ref{thm:szwarcnonnegII} or satisfy $c_n\leq1/2$ for other reasons. Recently, Kahler successfully applied Theorem~\ref{thm:szwarcnonnegII} to the large class of associated symmetric Pollaczek polynomials (with monotonicity of the whole sequence $(c_n)_{n\in\mathbb{N}}$), which is a two-parameter generalization of the well-known ultraspherical polynomials \cite{Ka21b}.\\

In \cite{Ka21a}, Kahler recently found the following example which, for certain choices of the parameters, satisfies nonnegative linearization of products without fulfilling the conditions of Theorem~\ref{thm:szwarcnonnegII}: for any $\alpha,\beta>-1$, let the sequence of generalized Chebyshev polynomials $(T_n^{(\alpha,\beta)}(x))_{n\in\mathbb{N}_0}\subseteq\mathbb{R}[x]$ be given by
\begin{equation*}
c_{2n-1}=\frac{n+\beta}{2n+\alpha+\beta},\;c_{2n}=\frac{n}{2n+\alpha+\beta+1}.
\end{equation*}
These polynomials are the quadratic transformations of the Jacobi polynomials and orthogonal w.r.t. $\mathrm{d}\mu(x)=\Gamma(\alpha+\beta+2)/(\Gamma(\alpha+1)\Gamma(\beta+1))\cdot(1-x^2)^{\alpha}|x|^{2\beta+1}\chi_{(-1,1)}(x)\,\mathrm{d}x$ (so $[-1,1]=\mathrm{supp}\;\mu=\widehat{\mathbb{N}_0}$) \cite[Chapter V 2 (G)]{Ch78} \cite[3 (f)]{La83}; for $\beta=-1/2$, one obtains the ultraspherical polynomials, including the Legendre polynomials and the Chebyshev polynomials of the first and second kind. The generalized Chebyshev polynomials are of particular interest concerning product formulas and duality structures \cite{La80,La83,Ob17}. In \cite[Theorem 3.2]{Ka21a}, Kahler showed that $(T_n^{(\alpha,\beta)}(x))_{n\in\mathbb{N}_0}$ satisfies nonnegative linearization of products if and only if $(\alpha,\beta)$ is an element of the set $V\subseteq[-1/2,\infty)\times(-1,\infty)$ given by
\begin{equation}\label{eq:Voriginal}
V:=\left\{(\alpha,\beta)\in(-1,\infty)^2:\alpha\geq\beta,a(a+5)(a+3)^2\geq(a^2-7a-24)b^2\right\},
\end{equation}
where $a:=\alpha+\beta+1$ and $b:=\alpha-\beta$.\footnote{This is the analogue to a well-known result of Gasper on the (nonsymmetric) class of Jacobi polynomials \cite[Theorem 1]{Ga70b}.} However, the conditions of Theorem~\ref{thm:szwarcnonnegII} are satisfied if and only if $\alpha\geq\beta$ and $\alpha+\beta+1\geq0$. If $(\alpha,\beta)\in V$ but $\alpha+\beta+1<0$,\footnote{Such pairs $(\alpha,\beta)$ exist, cf. Figure~\ref{fig:specialregiongencheb}; for instance, $(\alpha,\beta)=(-1/4,-5/6)$ has these properties.} then $(c_{2n})_{n\in\mathbb{N}}$ is strictly decreasing and always greater than $1/2$.\\

\begin{figure}
\centering
\includegraphics[width=0.60\textwidth]{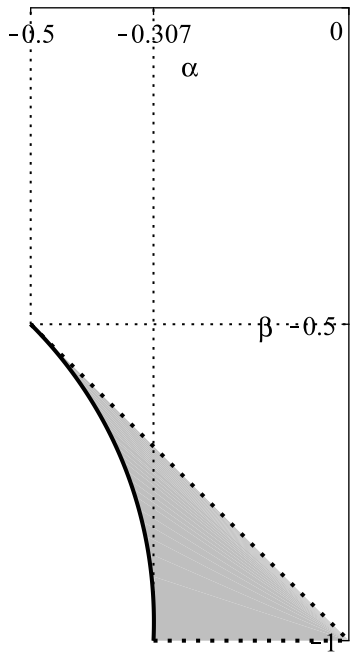}
\caption{The set $\{(\alpha,\beta)\in V:\alpha+\beta+1<0\}$, cf. \eqref{eq:Voriginal}.}\label{fig:specialregiongencheb}
\end{figure}

Nevertheless, the explicit formulas
\begin{align}
\label{eq:genchebhodd} h(2n-1)&=\frac{(\alpha+\beta+2)_{n-1}}{(\beta+1)_n}\cdot\frac{(2n+\alpha+\beta)(\alpha+1)_{n-1}}{(n-1)!}\\
\notag &=\left(2+\frac{\alpha-\beta}{n+\beta}\right)\prod_{k=0}^{n-2}\left[1+\frac{(2\alpha+1)k+\alpha^2+\alpha\beta+3\alpha+1}{(k+\beta+1)(k+1)}\right],\\
\label{eq:genchebheven} h(2n)&=\frac{(\alpha+\beta+2)_{n-1}}{(\beta+1)_n}\cdot\frac{(2n+\alpha+\beta+1)(\alpha+1)_n}{n!}\\
\notag &=\left(2+\frac{\alpha-\beta+\frac{(2\alpha+1)(n-1)+\alpha^2+\alpha\beta+3\alpha+1}{n}}{n+\beta}\right)\\
\notag &\quad\times\prod_{k=0}^{n-2}\left[1+\frac{(2\alpha+1)k+\alpha^2+\alpha\beta+3\alpha+1}{(k+\beta+1)(k+1)}\right]\;(n\in\mathbb{N})
\end{align}
\cite[3 (f)]{La83} and the estimation
\begin{equation}\label{eq:genchebhest}
\alpha^2+\alpha\beta+3\alpha+1\geq0
\end{equation}
show that $h(n)\geq2$ for all $n\in\mathbb{N}$. \eqref{eq:genchebhest} can be seen as follows: if $\alpha\geq0$, then
\begin{equation*}
\alpha^2+\alpha\beta+3\alpha+1\geq\alpha^2-\alpha+3\alpha+1=(\alpha+1)^2>0,
\end{equation*}
and if $\alpha<0$, then
\begin{equation*}
\alpha^2+\alpha\beta+3\alpha+1\geq\alpha^2+\alpha^2+3\alpha+1=(2\alpha+1)(\alpha+1)\geq0.
\end{equation*}
These observations yield the questions whether $h(n)\geq2\;(n\in\mathbb{N})$ is true for \textit{every} sequence $(P_n(x))_{n\in\mathbb{N}_0}$ which satisfies nonnegative linearization of products and whether maximal dual spaces $\widehat{\mathbb{N}_0}=[-1,1]$ (as satisfied by the generalized Chebyshev polynomials) play a more general role. In Section~\ref{sec:sufficient}, we give sufficient criteria (Theorem~\ref{thm:greaterthantwo}, which is indeed based on maximal dual spaces $\widehat{\mathbb{N}_0}=[-1,1]$, and the subsequent corollaries) which cover many naturally occurring examples, including the generalized Chebyshev polynomials (also those with $\alpha+\beta+1<0$ considered above). Concerning these criteria, we will discuss the role of nonnegative linearization of products, and we will consider the examples of cosh-polynomials and Grinspun polynomials. Moreover, in Section~\ref{sec:karlinmcgregor} we show that there are also counterexamples (Theorem~\ref{thm:smallerthantwo}). To our knowledge, these are the first examples with $h(1)<2$. It is easy to see that all such examples must fulfill $0\notin\widehat{\mathbb{N}_0}$. For every $\epsilon\in(0,1)$, we will construct polynomial hypergroups with $h(1)=1+\epsilon$. Our first corresponding example type (Theorem~\ref{thm:smallerthanoneplusepsilon}) relies on Karlin--McGregor polynomials, belongs to the class of Geronimus polynomials and has dual spaces of the form $\widehat{\mathbb{N}_0}=[-1,-1+\delta]\cup[1-\delta,1]$ with $\delta\in(0,1-\sqrt{(1-\epsilon)/(1+\epsilon)}]$; we will give explicit formulas for the corresponding orthogonalization measures, recurrence coefficients and Haar measures, and we will obtain an optimality result concerning the arising dual spaces. The problem under consideration is also interesting for the following reason: for the well-known Chebyshev polynomials of the first kind, which play a fundamental role in asymptotics and optimization, $h(n)$ \textit{equals} $2$ for all $n\in\mathbb{N}$. Hence, our results show that under a large class of naturally occurring examples the Chebyshev polynomials of the first kind are optimal w.r.t. minimizing the Haar function---however, they are not optimal among all possible examples. In Theorem~\ref{thm:smallerthanoneplusepsilon2}, we will obtain a second example type with $h(1)=1+\epsilon$; these examples do not rely on the Karlin--McGregor polynomials but on certain compact operators and convex sequences. They have discrete dual spaces of the form $\widehat{\mathbb{N}_0}=\{\pm1\}\cup\{\pm x_n:n\in\mathbb{N}\}$ with strictly increasing sequences $(x_n)_{n\in\mathbb{N}}\subseteq[\sqrt{(1-\epsilon)/(1+\epsilon)},1)$ with $\lim_{n\to\infty}x_n=1$, and they are of particular interest because $h$ grows exponentially but nevertheless the two dual objects $\widehat{\mathbb{N}_0}$ and $\mathrm{supp}\;\mu$ coincide. Section~\ref{sec:dual} is devoted to some additional notes on the dual space concerning a third dual object $\mathcal{X}^b(\mathbb{N}_0)$, which is homeomorphic to the (full) structure space $\Delta(\ell^1(h))$. Finally, Section~\ref{sec:openproblems} is devoted to some open problems.\\

We remark that we used computer algebra systems (Maple) to find suitable decompositions of long expressions, find explicit formulas, get conjectures and so on. The final proofs can be understood without any computer usage, however.

\section{Sufficient criteria for $h(n)\geq2\;(n\in\mathbb{N})$, the role of the dual space and the role of nonnegative linearization of products}\label{sec:sufficient}

In this section, we give some sufficient criteria for $h(n)\geq2\;(n\in\mathbb{N})$. They do not rely on boundedness properties of $(c_n)_{n\in\mathbb{N}}$, and they particularly cover examples where $(c_n)_{n\in\mathbb{N}}$ exceeds $1/2$ as considered in Section~\ref{sec:intro}. The dual space $\widehat{\mathbb{N}_0}$ will play a crucial role. Our approach is based on the connection coefficients to the Chebyshev polynomials of the first kind $(T_n(x))_{n\in\mathbb{N}_0}$: given an orthogonal polynomial sequence $(P_n(x))_{n\in\mathbb{N}_0}$ as in Section~\ref{sec:intro}, let $C_n(0),\ldots,C_n(n)$ be defined by the expansions
\begin{equation*}
P_n(x)=\sum_{k=0}^n C_n(k)T_k(x),
\end{equation*}
where $T_0(x)=1$, $T_1(x)=x$ and
\begin{equation}\label{eq:threetermreccheb}
x T_n(x)=\frac{1}{2}T_{n+1}(x)+\frac{1}{2}T_{n-1}(x)\;(n\in\mathbb{N})
\end{equation}
or, equivalently, $T_n(\cos(\varphi))=\cos(n\varphi)$. It is clear that $C_n(n)\neq0$. Recall that the Chebyshev polynomials of the first kind satisfy nonnegative linearization of products and $h(n)=2\;(n\in\mathbb{N})$; the orthogonalization measure $\mu$ is absolutely continuous (w.r.t. the Lebesgue--Borel measure on $\mathbb{R}$) and satisfies $\mathrm{d}\mu(x)=1/\pi\cdot(1-x^2)^{-1/2}\chi_{(-1,1)}(x)\,\mathrm{d}x$ \cite[Sect. 6]{La05}. We need the following classical estimation result from Chebyshev theory \cite[Theorem (3.1)]{To63}:

\begin{lemma}\label{lma:greaterthantwo}
Let $P(x)\in\mathbb{R}[x]$ be a polynomial of degree $n\in\mathbb{N}$ with leading coefficient $1$. Then
\begin{equation*}
\max_{x\in[-1,1]}|P(x)|\geq\frac{1}{2^{n-1}},
\end{equation*}
and equality holds if and only if $P(x)=T_n(x)/2^{n-1}$.
\end{lemma}

In the following, we always assume that $(P_n(x))_{n\in\mathbb{N}_0}$ satisfies nonnegative linearization of products. The following theorem is the central result of this section.

\begin{theorem}\label{thm:greaterthantwo}
Let the dual space $\widehat{\mathbb{N}_0}$ coincide with the full interval $[-1,1]$. Then $h(n)\geq2$ for all $n\in\mathbb{N}$.
\end{theorem}

\begin{proof}
Let $n\in\mathbb{N}\backslash\{1\}$ and expand $P_n(x)=\sum_{k=0}^n C_n(k)T_k(x)$. Since $\widehat{\mathbb{N}_0}=[-1,1]$, by Lemma~\ref{lma:greaterthantwo} we have
\begin{equation*}
1=\max_{x\in[-1,1]}|P_n(x)|=C_n(n)\max_{x\in[-1,1]}\left|\sum_{k=0}^n\frac{C_n(k)}{C_n(n)}T_k(x)\right|\geq C_n(n).
\end{equation*}
Since the leading coefficient of $P_n(x)$ is $1/\prod_{k=1}^{n-1}a_k$ and the leading coefficient of $T_n(x)$ is $2^{n-1}$, we get
\begin{equation*}
\frac{1}{\prod_{k=1}^{n-1}a_k}=C_n(n)\cdot2^{n-1}\leq2^{n-1}
\end{equation*}
and consequently
\begin{equation*}
4^{n-1}\prod_{k=1}^{n-1}a_k^2\geq1.
\end{equation*}
Using \eqref{eq:hrec}, we have
\begin{equation*}
h(n)=\frac{1}{c_1}\prod_{k=2}^n\frac{a_{k-1}}{c_k}=\frac{1}{c_n}\prod_{k=1}^{n-1}\frac{a_k}{c_k}=\frac{1}{c_n}\prod_{k=1}^{n-1}\frac{a_k^2}{c_k(1-c_k)}.
\end{equation*}
Since $c_k(1-c_k)\leq1/4$ for all $k\in\{1,\ldots,n-1\}$, we now obtain
\begin{equation*}
h(n)\geq\frac{1}{c_n}\cdot4^{n-1}\prod_{k=1}^{n-1}a_k^2\geq\frac{1}{c_n}.
\end{equation*}
Therefore, for every $n\in\mathbb{N}$ we have both
\begin{equation*}
1\leq c_n h(n)
\end{equation*}
(with equality for $n=1$) and
\begin{equation*}
1\leq c_{n+1}h(n+1)=a_n h(n),
\end{equation*}
so
\begin{equation*}
2\leq c_n h(n)+a_n h(n)=h(n).
\end{equation*}
\end{proof}

The essential condition $\widehat{\mathbb{N}_0}=[-1,1]$ in Theorem~\ref{thm:greaterthantwo} is fulfilled by an abundance of examples (see \cite{BH95,La83,La94,La05}, for instance). Some examples will be discussed later. We first give several corollaries.

\begin{corollary}\label{cor:greaterthantwoconcoeff}
If all connection coefficients $C_n(0),\ldots,C_n(n)$ are nonnegative, then $h(n)\geq2$ for all $n\in\mathbb{N}$.
\end{corollary}

\begin{proof}
As the connection coefficients $C_n(0),\ldots,C_n(n)$ sum up to $1$, the presumed nonnegativity allows to conclude in two ways: either obtain that $\widehat{\mathbb{N}_0}=[-1,1]$ as an immediate consequence and apply Theorem~\ref{thm:greaterthantwo}, or just use that the assumption particularly yields $C_n(n)\leq1$ and proceed as in the proof of Theorem~\ref{thm:greaterthantwo} above; the latter way avoids Lemma~\ref{lma:greaterthantwo}.
\end{proof}

\begin{corollary}\label{cor:bound}
If there exists a function $g:[-1,1]\rightarrow[0,\infty)$ such that $|P_n(x)|\leq g(x)$ for all $x\in[-1,1]$ and for all $n\in\mathbb{N}_0$, then $h(n)\geq2$ for all $n\in\mathbb{N}$.
\end{corollary}

\begin{proof}
It is a general result on polynomial hypergroups and their harmonic/functional analysis that the existence of such a function $g$ implies that $(P_n(x))_{n\in\mathbb{N}_0}$ is uniformly bounded on $[-1,1]$ by $\pm1$; one always has
\begin{equation}\label{eq:dualrelation}
\left\{x\in\mathbb{R}:\sup_{n\in\mathbb{N}_0}|P_n(x)|<\infty\right\}=\widehat{\mathbb{N}_0}
\end{equation}
\cite{La05}. Now Theorem~\ref{thm:greaterthantwo} yields the assertion.
\end{proof}

\begin{corollary}\label{cor:support}
If $\mathrm{supp}\;\mu=[-a,a]$ for some $a\in(0,1]$, then $h(n)\geq2$ for all $n\in\mathbb{N}$.
\end{corollary}

\begin{proof}
If $\mathrm{supp}\;\mu=[-a,a]$ for some $a\in(0,1]$, then $[-a,a]\subseteq\widehat{\mathbb{N}_0}$. Since the zeros of the polynomials $(P_n(x))_{n\in\mathbb{N}_0}$ are real, simple and located in $(-a,a)$, every $P_n(x)$ is nondecreasing on $[a,1]$. This shows that also $(a,1]\subseteq\widehat{\mathbb{N}_0}$. Finally, by symmetry we can conclude that $\widehat{\mathbb{N}_0}=[-1,1]$. Hence, the assertion follows from Theorem~\ref{thm:greaterthantwo}.
\end{proof}

\begin{corollary}\label{cor:greaterthantwo}
If $(c_n)_{n\in\mathbb{N}}$ is convergent, then $h(n)\geq2$ for all $n\in\mathbb{N}$.
\end{corollary}

\begin{proof}
If $(c_n)_{n\in\mathbb{N}}$ is convergent, then the limit $c$ is an element of $(0,1/2]$ and
\begin{equation*}
\mathrm{supp}\;\mu=[-2\sqrt{c(1-c)},2\sqrt{c(1-c)}]
\end{equation*}
due to \cite[Theorem (2.2)]{La94}, so the assertion follows from Corollary~\ref{cor:support}. Alternatively, one can obtain the result from Corollary~\ref{cor:greaterthantwoconcoeff}: by \cite[Theorem (2.6)]{La94} or \cite[Corollary 2]{LR93}, all connection coefficients $C_n(0),\ldots,C_n(n)$ are nonnegative.
\end{proof}

\begin{example}[cosh-polynomials]\label{example:cosh}
Let $a>0$, and let $(c_n)_{n\in\mathbb{N}}$ be given by
\begin{equation*}
c_n=\frac{\cosh(a(n-1))}{2\cosh(a n)\cosh(a)}.
\end{equation*}
$(c_n)_{n\in\mathbb{N}}$ is strictly decreasing, so Theorem~\ref{thm:szwarcnonnegII} cannot be applied. Nevertheless, the corresponding sequence $(P_n(x))_{n\in\mathbb{N}_0}$ satisfies nonnegative linearization of products (and therefore induces a polynomial hypergroup on $\mathbb{N}_0$); the linearization \eqref{eq:productlinear} takes the very simple form
\begin{equation*}
P_m(x)P_n(x)=\frac{\cosh(a(n-m))}{2\cosh(a m)\cosh(a n)}P_{n-m}(x)+\frac{\cosh(a(m+n))}{2\cosh(a m)\cosh(a n)}P_{m+n}(x)
\end{equation*}
for $n\geq m\geq 1$ \cite[Sect. 6]{La05}. The orthogonalization measure $\mu$ is absolutely continuous and satisfies
\begin{equation}\label{eq:measurecosh}
\mathrm{d}\mu(x)=\frac{1}{\pi}\cdot\left(\frac{1}{\cosh^2(a)}-x^2\right)^{-\frac{1}{2}}\chi_{\left(-\frac{1}{\cosh(a)},\frac{1}{\cosh(a)}\right)}(x)\,\mathrm{d}x
\end{equation}
\cite[Sect. 6]{La05}. Since
\begin{equation*}
\lim_{n\to\infty}c_n=\frac{1}{1+e^{2a}},
\end{equation*}
Corollary~\ref{cor:greaterthantwo} can be applied. For this easy example, the Haar weights grow exponentially and are also explicitly known with
\begin{equation*}
h(n)=\begin{cases} 1, & n=0, \\ 2\cosh^2(a n), & n\in\mathbb{N}, \end{cases}
\end{equation*}
see \cite[Sect. 6]{La05}. The desired estimation $h(n)\geq2\;(n\in\mathbb{N})$ also follows just from the boundedness of $(c_n)_{n\in\mathbb{N}}$ by $1/2$. The limiting case $a=0$ corresponds to the Chebyshev polynomials of the first kind. Moreover, it is easy to see from \eqref{eq:measurecosh} that
\begin{equation*}
T_n(x)=\frac{P_n\left(\frac{x}{\cosh(a)}\right)}{P_n\left(\frac{1}{\cosh(a)}\right)}
\end{equation*}
for every $n\in\mathbb{N}_0$; in other words: if one rescales the cosh-polynomials in such a way that the right endpoint of the support of the measure becomes $1$, then one obtains the Chebyshev polynomials of the first kind (up to renormalization), independently from the parameter $a$. A similar procedure for a less trivial example will occur and be crucial in Section~\ref{sec:karlinmcgregor} in order to construct examples with $h(1)<2$.
\end{example}

Considering the orthonormal polynomials $(p_n(x))_{n\in\mathbb{N}_0}$ and following Nevai \cite{Ne79}, one says $\mu\in M(0,b)$, $b\in(0,1]$, if $\lim_{n\to\infty}\alpha_n=b/2$ or, equivalently\footnote{The implication $\lim_{n\to\infty}\alpha_n=b/2\Rightarrow\lim_{n\to\infty}c_n=(1-\sqrt{1-b^2})/2$ is not obvious and can be seen as follows: let $\lim_{n\to\infty}\alpha_n=b/2$. Results on chain sequences \cite[Theorem III-6.4]{Ch78} \cite[Proposition 5]{Sz94a} yield that $(c_n)_{n\in\mathbb{N}}$ converges to $(1-\sqrt{1-b^2})/2$ or $(1+\sqrt{1-b^2})/2$. Since $(c_n)_{n\in\mathbb{N}}$ cannot converge to a value which is strictly greater than $1/2$ \cite[Theorem (2.2)]{La94}, we have $\lim_{n\to\infty}c_n=(1-\sqrt{1-b^2})/2$.},
\begin{equation*}
\lim_{n\to\infty}c_n=\frac{1}{2}(1-\sqrt{1-b^2}).
\end{equation*}
Many naturally occurring examples satisfy $\mu\in M(0,1)$. The cosh-polynomials considered in Example~\ref{example:cosh} satisfy $\mu\in M(0,b)$ with $b<1$.

\begin{corollary}\label{cor:nevai}
If $\mu\in M(0,b)$ for some $b\in(0,1]$, then $h(n)\geq2$ for all $n\in\mathbb{N}$.
\end{corollary}

\begin{proof}
This a reformulation of Corollary~\ref{cor:greaterthantwo}.
\end{proof}

\begin{example}[generalized Chebyshev polynomials reconsidered]\label{example:gencheb}
For all $(\alpha,\beta)\in V$, including those pairs where $\alpha+\beta+1<0$ (cf. Figure~\ref{fig:specialregiongencheb}) and $(c_{2n})_{n\in\mathbb{N}}$ exceeds $1/2$, the generalized Chebyshev polynomials $(T_n^{(\alpha,\beta)}(x))_{n\in\mathbb{N}_0}$ fulfill the assumptions of Theorem~\ref{thm:greaterthantwo} and Corollary~\ref{cor:greaterthantwoconcoeff} to Corollary~\ref{cor:nevai} (the latter with $\mu\in M(0,1)$).
\end{example}

\begin{remark}
The estimation $h(n)\geq2\;(n\in\mathbb{N})$, which is fulfilled with equality for the Chebyshev polynomials of the first kind, can be interpreted in the following sense: under the conditions of Theorem~\ref{thm:greaterthantwo}, the Chebyshev polynomials of the first kind are optimal w.r.t. minimizing the Haar function. Comparing this to the optimality assertion of Lemma~\ref{lma:greaterthantwo}, one might ask the question whether, under the conditions of Theorem~\ref{thm:greaterthantwo} (or the subsequent corollaries), $h(n)$ is strictly greater than $2$ for all $n\in\mathbb{N}$ as soon as $(P_n(x))_{n\in\mathbb{N}_0}\neq(T_n(x))_{n\in\mathbb{N}_0}$. However, this is not the case: for every $\alpha>-1/2$, the generalized Chebyshev polynomials $(T_n^{(\alpha,\alpha)}(x))_{n\in\mathbb{N}_0}\neq(T_n(x))_{n\in\mathbb{N}_0}$ satisfy nonnegative linearization of products, as well as the conditions of Theorem~\ref{thm:greaterthantwo} and of the subsequent corollaries (cf. Example~\ref{example:gencheb}), but $h(1)=2$ (cf. \eqref{eq:genchebhodd} and Figure~\ref{fig:hgenchebmeasures}).
\end{remark}

\begin{figure}
\centering
\includegraphics[width=0.60\textwidth]{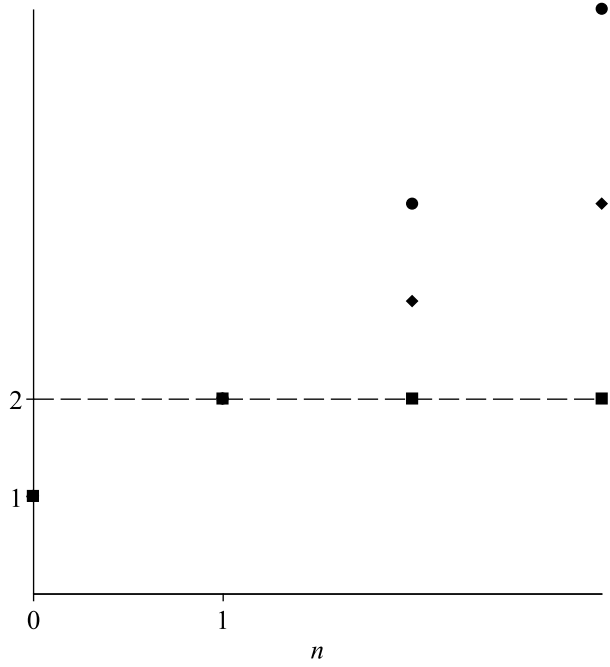}
\caption{The first Haar weights $h(n)$ belonging to $(T_n^{(\alpha,\alpha)}(x))_{n\in\mathbb{N}_0}$, cf. \eqref{eq:genchebhodd} and \eqref{eq:genchebheven}, for $\alpha=-1/2$ (box, corresponds to the Chebyshev polynomials of the first kind $(T_n(x))_{n\in\mathbb{N}_0}$), $\alpha=0$ (diamond) and $\alpha=1/2$ (circle). In all cases, one has $h(1)=2$.}\label{fig:hgenchebmeasures}
\end{figure}

Note that the formal definitions of $\widehat{\mathbb{N}_0}$ \eqref{eq:dualdef} and $h$ \eqref{eq:hdef} also make sense if nonnegative linearization of products is not satisfied (and hence without the underlying hypergroup structure). It is obvious that still $\{\pm1\}\subseteq\widehat{\mathbb{N}_0}\subseteq[-1,1]$. However, the property $\mathrm{supp}\;\mu\subseteq\widehat{\mathbb{N}_0}$, which is a consequence of harmonic/functional analysis on polynomial hypergroups, does no longer have to be satisfied. \eqref{eq:hrec} remains true, but $h$ can now map into the larger codomain $(0,\infty)$. The rest of the section is devoted to the question which of our results remain true under these more general conditions. From a harmonic/functional analytic point of view, such considerations are interesting because reasonable generalizations of the hypergroup structure exist under weaker conditions than nonnegative linearization of products \cite{LOR07,OW96}.\\

The proof of Theorem~\ref{thm:greaterthantwo} remains fully true if the nonnegative linearization of products condition is dropped, as well as the proof of Corollary~\ref{cor:greaterthantwoconcoeff} which is based on nonnegativity of the connection coefficients $C_n(0),\ldots,C_n(n)$. There are rather general sufficient criteria for the nonnegativity of all connection coefficients $C_n(0),\ldots,C_n(n)$. In particular, the boundedness of $(c_n)_{n\in\mathbb{N}}$ from above by $1/2$ is sufficient \cite[Corollary 1]{Sz92c} (however, recall that the desired estimation $h(n)\geq2\;(n\in\mathbb{N})$ is trivial in that case).\\

The following example shows that the further corollaries do not extend if the nonnegative linearization of products condition is dropped (and therefore the tool of harmonic/functional analysis on polynomial hypergroups is no longer available), however.

\begin{example}[Grinspun polynomials]
Let $c_1\in(0,1)$ be arbitrary, and let $c_n=1/2$ for every $n\geq2$. The resulting polynomials $(P_n(x))_{n\in\mathbb{N}_0}$ are the Grinspun polynomials and orthogonal w.r.t. a measure $\mu$ with $\mathrm{supp}\;\mu=[-1,1]$ \cite[Chapter VI 13 (C) (iv)]{Ch78}. It is clear that $\mu\in M(0,1)$. Via induction and \eqref{eq:threetermreccheb}, it is easy to see that
\begin{equation}\label{eq:grinspunexpand}
P_n(x)=\frac{1}{2-2c_1}T_n(x)+\frac{1-2c_1}{2-2c_1}T_{n-2}(x)\;(n\geq2)
\end{equation}
and therefore
\begin{equation}\label{eq:grinspunexpandmod}
P_n(x)=T_n(x)+\frac{2c_1-1}{2-2c_1}(T_n(x)-T_{n-2}(x))\;(n\geq2)
\end{equation}
(cf. also \cite[VI-(13.9)]{Ch78} and \cite[Section 3.2]{Vo93}). If $c_1\leq1/2$, then nonnegative linearization of products is satisfied (and hence a polynomial hypergroup on $\mathbb{N}_0$ is induced) \cite[3 (g) (ii)]{La83}. If $c_1>1/2$, however, then nonnegative linearization of products fails and the expansions \eqref{eq:grinspunexpand} imply that $(P_n(x))_{n\in\mathbb{N}_0}$ is uniformly bounded on $[-1,1]$ by $\pm c_1/(1-c_1)$, but \eqref{eq:hrec} yields
\begin{equation*}
h(1)=\frac{1}{c_1}<2
\end{equation*}
and
\begin{equation*}
h(n)=2\frac{1-c_1}{c_1}<2\;(n\geq2).
\end{equation*}
This shows that Corollary~\ref{cor:bound} is not valid without nonnegative linearization of products; if $c_1>2/3$, then $h$ does not even map to $[1,\infty)$. Reconsidering the proof of Corollary~\ref{cor:bound}, we see that \eqref{eq:dualrelation} made use of nonnegative linearization of products (and of the resulting harmonic/functional analysis due to the hypergroup aspect). Clearly, the example also shows that Corollary~\ref{cor:support}, Corollary~\ref{cor:greaterthantwo} and Corollary~\ref{cor:nevai} are not valid without nonnegative linearization of products. It is already clear from the preceding considerations that neither Corollary~\ref{cor:greaterthantwoconcoeff} nor Theorem~\ref{thm:greaterthantwo} can apply to the case $c_1>1/2$, and one can see from \eqref{eq:grinspunexpand} and \eqref{eq:grinspunexpandmod} at which stages an application exactly fails: let $c_1>1/2$. Then \eqref{eq:grinspunexpand} yields that $C_n(n-2)<0$ for all $n\geq2$. Moreover, one has $\widehat{\mathbb{N}_0}=\{\pm1\}$, which can be seen as follows: let $x\in(-1,1)$ and $\varphi\in(0,\pi)$ with $x=\cos(\varphi)$. Then, by \eqref{eq:grinspunexpandmod},
\begin{align*}
P_n(x)&=T_n(\cos(\varphi))+\frac{2c_1-1}{2-2c_1}(T_n(\cos(\varphi))-T_{n-2}(\cos(\varphi)))\\
&=\cos(n\varphi)+\frac{2c_1-1}{2-2c_1}(\cos(n\varphi)-\cos((n-2)\varphi))\\
&=\cos(n\varphi)+\frac{2c_1-1}{2-2c_1}(1-\cos(2\varphi))\cos(n\varphi)-\frac{2c_1-1}{2-2c_1}\sin(2\varphi)\sin(n\varphi)\\
\end{align*}
for every $n\geq2$. Now let $(n_k)_{k\in\mathbb{N}}\subseteq\mathbb{N}\backslash\{1\}$ be a sequence with $\lim_{k\to\infty}\cos(n_k\varphi)=1$ (and consequently $\lim_{k\to\infty}\sin(n_k\varphi)=0$). Then
\begin{equation*}
\lim_{k\to\infty}P_{n_k}(x)=1+\frac{2c_1-1}{2-2c_1}(1-\cos(2\varphi))>1
\end{equation*}
and we can conclude that $x\notin\widehat{\mathbb{N}_0}$.
\end{example}

\section{Two types of examples which do not satisfy $h(n)\geq2\;(n\in\mathbb{N})$ and properties of their dual spaces}\label{sec:karlinmcgregor}

Having seen sufficient criteria for $h(n)\geq2\;(n\in\mathbb{N})$ in the previous section, we now construct examples where nonnegative linearization of products is satisfied but $h(1)<2$. Moreover, we deal with necessary criteria concerning the latter property. We begin with the following observation concerning the dual space $\widehat{\mathbb{N}_0}$.

\begin{proposition}\label{prp:dualspace}
If $h(1)=1+\epsilon$ with $\epsilon\in(0,1)$, then
\begin{equation}\label{eq:crucialnecessarysharpening}
\widehat{\mathbb{N}_0}\subseteq\left[-1,-\sqrt{\frac{1-\epsilon}{1+\epsilon}}\right]\cup\left[\sqrt{\frac{1-\epsilon}{1+\epsilon}},1\right].
\end{equation}
\end{proposition}

\begin{proof}
If $h(1)=1+\epsilon$ with $\epsilon\in(0,1)$, then
\begin{equation*}
P_2(x)=\frac{x^2-c_1}{1-c_1}=\frac{h(1)x^2-1}{h(1)-1}=\frac{(1+\epsilon)x^2-1}{\epsilon},
\end{equation*}
so
\begin{equation*}
P_2\left(\pm\sqrt{\frac{1-\epsilon}{1+\epsilon}}\right)=-1.
\end{equation*}
Therefore, we have $P_2(x)<-1$ for $x\in(-\sqrt{(1-\epsilon)/(1+\epsilon)},\sqrt{(1-\epsilon)/(1+\epsilon)})$, which yields the assertion.
\end{proof}

\begin{remark}\label{rem:dualspace}
As a much less trivial result, in Theorem~\ref{thm:smallerthanoneplusepsilon} we will obtain that there are examples which satisfy \eqref{eq:crucialnecessarysharpening} with \textit{equality}; so the estimation provided by Proposition~\ref{prp:dualspace} cannot be improved.
\end{remark}

As an immediate consequence of Proposition~\ref{prp:dualspace}, all examples with $h(1)<2$ have to share the necessary condition
\begin{equation}\label{eq:crucialnecessary}
0\notin\widehat{\mathbb{N}_0}.
\end{equation}

We use the Karlin--McGregor polynomials as a starting point in order to construct polynomial hypergroups with $h(1)<2$. For $\alpha,\beta\geq2$, the Karlin--McGregor polynomials $(K_n^{(\alpha,\beta)}(x))_{n\in\mathbb{N}_0}\subseteq\mathbb{R}[x]$ are given by
\begin{equation*}
c_{2n-1}=\frac{1}{\alpha}
\end{equation*}
and
\begin{equation*}
c_{2n}=\frac{1}{\beta}
\end{equation*}
\cite[Sect. 6]{La05}. For any choice of $\alpha,\beta\geq2$, $(K_n^{(\alpha,\beta)}(x))_{n\in\mathbb{N}_0}$ fulfills the conditions of Theorem~\ref{thm:szwarcnonnegII}, so nonnegative linearization of products is always satisfied and $h(n)\geq2\;(n\in\mathbb{N})$. If $\alpha=\beta=2$, one obtains the Chebyshev polynomials of the first kind (which played a crucial role in Section~\ref{sec:sufficient}). One has $\widehat{\mathbb{N}_0}=[-1,1]$ \cite[Sect. 6]{La05}, so the property $h(n)\geq2\;(n\in\mathbb{N})$ is also a consequence of our sufficient criterion Theorem~\ref{thm:greaterthantwo}. Nevertheless, a modification of the Karlin--McGregor polynomials will yield examples which fulfill the necessary condition \eqref{eq:crucialnecessary} and even the desired property $h(1)<2$ (see Theorem~\ref{thm:smallerthantwo}, the subsequent remarks and Theorem~\ref{thm:smallerthanoneplusepsilon} below). The basic underlying idea is that there are choices of $\alpha,\beta\geq2$ such that $(K_n^{(\alpha,\beta)}(x))_{n\in\mathbb{N}_0}$ satisfies the condition
\begin{equation}\label{eq:crucialnecessaryprime}
0\notin\mathrm{supp}\;\mu
\end{equation}
(cf. \eqref{eq:supportmuKarlinMcGregor} below). As nonnegative linearization of products yields $\mathrm{supp}\;\mu\subseteq\widehat{\mathbb{N}_0}$, \eqref{eq:crucialnecessaryprime} is a weaker condition than \eqref{eq:crucialnecessary}. In order to obtain the stronger condition \eqref{eq:crucialnecessary}, our strategy will be to modify the polynomials in such a way that \eqref{eq:crucialnecessaryprime} remains true and $\mathrm{supp}\;\mu$ coincides with the dual space $\widehat{\mathbb{N}_0}$.\\

We first recall some basics about the Karlin--McGregor polynomials. If $\alpha\leq\beta$, then the orthogonalization measure $\mu$ is absolutely continuous and satisfies
\begin{equation}\label{eq:measureKarlinMcGregorabscont}
\mathrm{d}\mu(x)=\begin{cases} \frac{\beta\sqrt{(\gamma_1^2-x^2)(x^2-\gamma_2^2)}}{2\pi|x|(1-x^2)}\chi_{(-\gamma_1,-\gamma_2)\cup(\gamma_2,\gamma_1)}(x)\,\mathrm{d}x, & \alpha<\beta, \\ \frac{\alpha\sqrt{\gamma_1^2-x^2}}{2\pi(1-x^2)}\chi_{(-\gamma_1,\gamma_1)}(x)\,\mathrm{d}x, & \alpha=\beta, \end{cases}
\end{equation}
where
\begin{equation*}
\gamma_1:=\frac{1}{\sqrt{\alpha\beta}}(\sqrt{\alpha-1}+\sqrt{\beta-1})
\end{equation*}
and
\begin{equation*}
\gamma_2:=\frac{1}{\sqrt{\alpha\beta}}|\sqrt{\alpha-1}-\sqrt{\beta-1}|.
\end{equation*}
If $\alpha>\beta$, then $\mu$ consists of an absolutely continuous part given by
\begin{equation}\label{eq:measureKarlinMcGregorabscontpart}
\frac{\beta\sqrt{(\gamma_1^2-x^2)(x^2-\gamma_2^2)}}{2\pi|x|(1-x^2)}\chi_{(-\gamma_1,-\gamma_2)\cup(\gamma_2,\gamma_1)}(x)\,\mathrm{d}x
\end{equation}
and a discrete part given by the point mass
\begin{equation}\label{eq:measureKarlinMcGregordiscr}
\mu(\{0\})=\frac{\alpha-\beta}{\alpha}.
\end{equation}
Essentially, the preceding formulas for $\mu$ can be found in Karlin and McGregor's paper \cite{KM59}, where a slightly different system of orthogonal polynomials was considered. Moreover, the first part of \eqref{eq:measureKarlinMcGregorabscont} can be found in \cite[p. 49, p. 108]{Wa14} (without reference or explanation), and \eqref{eq:measureKarlinMcGregordiscr} can also be found in \cite{FL00}. Up to a small mistake, the formulas for $\mu$ and a proof via Stieltjes transforms can also be found in \cite[Theorem 4.2]{Eh94}. Since the deduction of \eqref{eq:measureKarlinMcGregorabscont} to \eqref{eq:measureKarlinMcGregordiscr} from the polynomials originally considered in Karlin and McGregor's paper \cite{KM59} is not obvious (and since the relevant passage of \cite{KM59} contains a small mistake, too), we will give a proof sketch in the appendix. For the moment, we note that one particularly has
\begin{equation}\label{eq:supportmuKarlinMcGregor}
\mathrm{supp}\;\mu=\begin{cases} [-\gamma_1,-\gamma_2]\cup[\gamma_2,\gamma_1], & \alpha\leq\beta, \\ [-\gamma_1,-\gamma_2]\cup\{0\}\cup[\gamma_2,\gamma_1], & \alpha>\beta. \end{cases}
\end{equation}
It is obvious from \eqref{eq:hrec} that the Haar weights are given by $h(0)=1$ and
\begin{equation}\label{eq:karlinmcgregorhodd}
h(2n-1)=\alpha(\alpha-1)^{n-1}(\beta-1)^{n-1}
\end{equation}
and
\begin{equation}\label{eq:karlinmcgregorheven}
h(2n)=\beta(\alpha-1)^n(\beta-1)^{n-1}
\end{equation}
for $n\in\mathbb{N}$ \cite[Sect. 6]{La05}. Moreover, it is easy to see via induction that
\begin{equation}\label{eq:karlinmcgregorgamma1even}
K_{2n}^{(\alpha,\beta)}(\gamma_1)=\frac{(\alpha-2)\sqrt{\beta-1}+(\beta-2)\sqrt{\alpha-1}}{\beta(\alpha-1)^{\frac{n+1}{2}}(\beta-1)^{\frac{n}{2}}}\cdot n+\frac{1}{(\alpha-1)^{\frac{n}{2}}(\beta-1)^{\frac{n}{2}}}
\end{equation}
and
\begin{equation}\label{eq:karlinmcgregorgamma1odd}
K_{2n+1}^{(\alpha,\beta)}(\gamma_1)=\frac{(\alpha-2)\sqrt{\beta-1}+(\beta-2)\sqrt{\alpha-1}}{\sqrt{\alpha\beta}(\alpha-1)^{\frac{n+1}{2}}(\beta-1)^{\frac{n+1}{2}}}\cdot n+\frac{\sqrt{\alpha-1}+\sqrt{\beta-1}}{\sqrt{\alpha\beta}(\alpha-1)^{\frac{n}{2}}(\beta-1)^{\frac{n}{2}}}
\end{equation}
for all $n\in\mathbb{N}_0$. Concerning the special case $\alpha=2$, it was shown in \cite[p. 72]{Eh94} via sieved polynomials that
\begin{equation}\label{eq:kmspecialodd}
K_{2n-1}^{(2,\beta)}(x)=\frac{1}{(\beta-1)^{\frac{n}{2}}}x\left(\sqrt{\beta-1}U_{n-1}\left(\frac{(2x^2-1)\beta}{2\sqrt{\beta-1}}\right)-U_{n-2}\left(\frac{(2x^2-1)\beta}{2\sqrt{\beta-1}}\right)\right)
\end{equation}
and
\begin{equation}\label{eq:kmspecialeven}
K_{2n}^{(2,\beta)}(x)=\frac{1}{(\beta-1)^{\frac{n}{2}}}\left(\frac{\beta-1}{\beta}U_n\left(\frac{(2x^2-1)\beta}{2\sqrt{\beta-1}}\right)-\frac{1}{\beta}U_{n-2}\left(\frac{(2x^2-1)\beta}{2\sqrt{\beta-1}}\right)\right)
\end{equation}
for every $n\in\mathbb{N}$, where $(U_n(x))_{n\in\mathbb{N}_0}$ denotes the sequence of Chebyshev polynomials of the second kind (i.e., $U_n(\cos(\varphi))\sin(\varphi)=\sin((n+1)\varphi)$) with the convention $U_{-1}(x):=0$.\\

The following result provides our first---and simplest---example with $h(1)<2$.

\begin{theorem}\label{thm:smallerthantwo}
Let
\begin{equation*}
c_{2n-1}=\frac{6n+4}{9n+9}\\
\end{equation*}
and
\begin{equation*}
c_{2n}=\frac{n+1}{3n+5}.
\end{equation*}
Then $(P_n(x))_{n\in\mathbb{N}_0}$ satisfies nonnegative linearization of products, i.e., $(P_n(x))_{n\in\mathbb{N}_0}$ induces a polynomial hypergroup on $\mathbb{N}_0$. The Haar function $h:\mathbb{N}_0\rightarrow[1,\infty)$ satisfies
\begin{equation*}
h(2n-1)=\frac{9}{5}\left(\frac{n}{2}+\frac{1}{2}\right)^2
\end{equation*}
and
\begin{equation*}
h(2n)=\frac{9}{5}\left(\frac{n}{2}+\frac{5}{6}\right)^2
\end{equation*}
for every $n\in\mathbb{N}$. In particular, one has $h(1)=9/5<2$.
\end{theorem}

\begin{proof}
Let $(Q_n(x))_{n\in\mathbb{N}_0}\subseteq\mathbb{R}[x]$ be defined by
\begin{equation*}
Q_{2n}(x)=\frac{3n+5}{2^n\cdot5}P_{2n}\left(\frac{\sqrt{10}}{3}x\right)
\end{equation*}
and
\begin{equation*}
Q_{2n+1}(x)=\frac{3n+6}{2^{n+1}\sqrt{10}}P_{2n+1}\left(\frac{\sqrt{10}}{3}x\right).
\end{equation*}
We \textit{claim} that $(Q_n(x))_{n\in\mathbb{N}_0}$ coincides with the Karlin--McGregor polynomials $(K_n^{(2,5)}(x))_{n\in\mathbb{N}_0}$. Once the claim is established, we obtain the assertions as follows: since $(K_n^{(2,5)}(x))_{n\in\mathbb{N}_0}$ satisfies nonnegative linearization of products, $(P_n(x))_{n\in\mathbb{N}_0}$ satisfies nonnegative linearization of products, too; the explicit formulas for $h$ are easily verified by \eqref{eq:hrec} and induction.\\

$Q_0(x)=1=K_0^{(2,5)}(x)$ and $Q_1(x)=x=K_1^{(2,5)}(x)$ are immediate from the definitions. Let $n\in\mathbb{N}_0$ be arbitrary but fixed, and assume that $Q_{2n}(x)=K_{2n}^{(2,5)}(x)$ and $Q_{2n+1}(x)=K_{2n+1}^{(2,5)}(x)$. Then
\begin{align*}
&K_{2n+2}^{(2,5)}(x)\\
&=\frac{x K_{2n+1}^{(2,5)}(x)-\frac{1}{2}K_{2n}^{(2,5)}(x)}{\frac{1}{2}}\\
&=\frac{x\cdot\frac{3n+6}{2^{n+1}\sqrt{10}}P_{2n+1}\left(\frac{\sqrt{10}}{3}x\right)-\frac{1}{2}\cdot\frac{3n+5}{2^n\cdot5}P_{2n}\left(\frac{\sqrt{10}}{3}x\right)}{\frac{1}{2}}\\
&=\frac{\frac{3n+6}{2^{n+1}\sqrt{10}}\cdot\frac{3}{\sqrt{10}}\left(\frac{3n+8}{9n+18}P_{2n+2}\left(\frac{\sqrt{10}}{3}x\right)+\frac{6n+10}{9n+18}P_{2n}\left(\frac{\sqrt{10}}{3}x\right)\right)-\frac{1}{2}\cdot\frac{3n+5}{2^n\cdot5}P_{2n}\left(\frac{\sqrt{10}}{3}x\right)}{\frac{1}{2}}\\
&=\frac{3n+8}{2^{n+1}\cdot5}P_{2n+2}\left(\frac{\sqrt{10}}{3}x\right)\\
&=Q_{2n+2}(x)
\end{align*}
and
\begin{align*}
&K_{2n+3}^{(2,5)}(x)\\
&=\frac{x K_{2n+2}^{(2,5)}(x)-\frac{1}{5}K_{2n+1}^{(2,5)}(x)}{\frac{4}{5}}\\
&=\frac{x\cdot\frac{3n+8}{2^{n+1}\cdot5}P_{2n+2}\left(\frac{\sqrt{10}}{3}x\right)-\frac{1}{5}\cdot\frac{3n+6}{2^{n+1}\sqrt{10}}P_{2n+1}\left(\frac{\sqrt{10}}{3}x\right)}{\frac{4}{5}}\\
&=\frac{\frac{3n+8}{2^{n+1}\cdot5}\cdot\frac{3}{\sqrt{10}}\left(\frac{2n+6}{3n+8}P_{2n+3}\left(\frac{\sqrt{10}}{3}x\right)+\frac{n+2}{3n+8}P_{2n+1}\left(\frac{\sqrt{10}}{3}x\right)\right)-\frac{1}{5}\cdot\frac{3n+6}{2^{n+1}\sqrt{10}}P_{2n+1}\left(\frac{\sqrt{10}}{3}x\right)}{\frac{4}{5}}\\
&=\frac{3n+9}{2^{n+2}\sqrt{10}}P_{2n+3}\left(\frac{\sqrt{10}}{3}x\right)\\
&=Q_{2n+3}(x).
\end{align*}
\end{proof}

The dual space for the polynomials considered in Theorem~\ref{thm:smallerthantwo} will be obtained later.\\

Reconsidering the proof of Theorem~\ref{thm:smallerthantwo} and following the strategy outlined above, we can find further examples which satisfy $h(1)<2$ or even $h(1)=1+\epsilon$ with $\epsilon\in(0,1)$ (and nonnegative linearization of products). We start with the Karlin--McGregor polynomials $(K_n^{(\alpha,\beta)}(x))_{n\in\mathbb{N}_0}$, $\alpha,\beta\geq2$. Next, we rescale them in such a way that the right endpoint of the support of the measure becomes $1$. Finally, we renormalize the resulting polynomials in such a way that $P_n(1)\equiv1$ again. This procedure ends up in the sequence $(P_n(x))_{n\in\mathbb{N}_0}\subseteq\mathbb{R}[x]$ of modified Karlin--McGregor polynomials given by
\begin{equation*}
P_n(x)=\frac{K_n^{(\alpha,\beta)}(\gamma_1x)}{K_n^{(\alpha,\beta)}(\gamma_1)},
\end{equation*}
and $(P_n(x))_{n\in\mathbb{N}_0}$ still satisfies nonnegative linearization of products. The above-mentioned further examples will be obtained below for suitable choices of $\alpha$ and $\beta$. We first study the polynomials $(P_n(x))_{n\in\mathbb{N}_0}$ in detail and compute the associated orthogonalization measures, Haar measures and recurrence coefficients.\\

Using \eqref{eq:measureKarlinMcGregorabscont} to \eqref{eq:measureKarlinMcGregordiscr}, we find that $(P_n(x))_{n\in\mathbb{N}_0}$ is orthogonal w.r.t. the probability measure $\mu$ satisfying the following properties: if $\alpha\leq\beta$, then $\mu$ is absolutely continuous and fulfills
\begin{equation}\label{eq:measuremodifiedKarlinMcGregorabscont}
\mathrm{d}\mu(x)=\begin{cases} \frac{\beta\gamma_1\sqrt{(1-x^2)(\gamma_1^2x^2-\gamma_2^2)}}{2\pi|x|(1-\gamma_1^2x^2)}\chi_{\left(-1,-\frac{\gamma_2}{\gamma_1}\right)\cup\left(\frac{\gamma_2}{\gamma_1},1\right)}(x)\,\mathrm{d}x, & \alpha<\beta, \\ \frac{\alpha\gamma_1^2\sqrt{1-x^2}}{2\pi(1-\gamma_1^2x^2)}\chi_{(-1,1)}(x)\,\mathrm{d}x, & \alpha=\beta. \end{cases}
\end{equation}
If $\alpha>\beta$, however, then $\mu$ consists of an absolutely continuous part given by
\begin{equation}\label{eq:measuremodifiedKarlinMcGregorabscont2}
\frac{\beta\gamma_1\sqrt{(1-x^2)(\gamma_1^2x^2-\gamma_2^2)}}{2\pi|x|(1-\gamma_1^2x^2)}\chi_{\left(-1,-\frac{\gamma_2}{\gamma_1}\right)\cup\left(\frac{\gamma_2}{\gamma_1},1\right)}(x)\,\mathrm{d}x
\end{equation}
and a discrete part given by
\begin{equation}\label{eq:measuremodifiedKarlinMcGregordisc}
\mu(\{0\})=\frac{\alpha-\beta}{\alpha}.
\end{equation}
Plots for different values of $(\alpha,\beta)$ can be found in Figure~\ref{fig:orthmeasures}.\\

\begin{figure}
\centering
\includegraphics[width=1.00\textwidth]{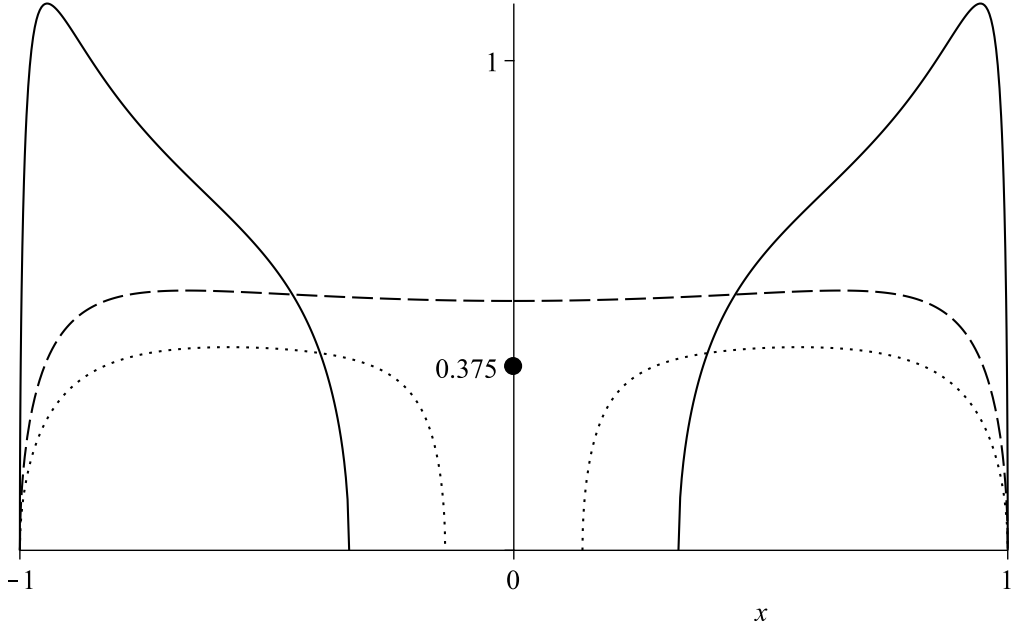}
\caption{The absolutely continuous part of $\mu$, cf. \eqref{eq:measuremodifiedKarlinMcGregorabscont} and \eqref{eq:measuremodifiedKarlinMcGregorabscont2}, for $(\alpha,\beta)=(2,5)$ (solid line), $(\alpha,\beta)=(5,5)$ (dashed line) and $(\alpha,\beta)=(8,5)$ (dotted line). In the letter case, there is an additional discrete part with $\mu(\{0\})=3/8$, cf. \eqref{eq:measuremodifiedKarlinMcGregordisc}.}\label{fig:orthmeasures}
\end{figure}

By construction, the Haar weights corresponding to the modified polynomials $(P_n(x))_{n\in\mathbb{N}_0}$ and the Haar weights corresponding to the original Karlin--McGregor polynomials $(K_n^{(\alpha,\beta)}(x))_{n\in\mathbb{N}_0}$ are linked to each other by multiplication with $(K_n^{(\alpha,\beta)}(\gamma_1))^2$. Using \eqref{eq:karlinmcgregorhodd} to \eqref{eq:karlinmcgregorgamma1odd}, we obtain that the Haar weights associated with $(P_n(x))_{n\in\mathbb{N}_0}$ satisfy $h(0)=1$ and
\begin{equation}\label{eq:modifiedkarlinmcgregorhodd}
h(2n-1)=\frac{1}{\beta}\left[\left(\frac{\alpha-2}{\sqrt{\alpha-1}}+\frac{\beta-2}{\sqrt{\beta-1}}\right)\cdot(n-1)+\sqrt{\alpha-1}+\sqrt{\beta-1}\right]^2
\end{equation}
and
\begin{equation}\label{eq:modifiedkarlinmcgregorheven}
h(2n)=\frac{1}{\beta}\left[\left(\frac{\alpha-2}{\sqrt{\alpha-1}}+\frac{\beta-2}{\sqrt{\beta-1}}\right)\cdot n+\frac{\beta}{\sqrt{\beta-1}}\right]^2
\end{equation}
for $n\in\mathbb{N}$. Observe that $h$ is always of quadratic (and therefore subexponential) growth, which particularly implies that $\widehat{\mathbb{N}_0}=\mathrm{supp}\;\mu$ now as desired. Plots can be found in Figure~\ref{fig:hmeasures}.\\

Via \eqref{eq:hrec}, \eqref{eq:modifiedkarlinmcgregorhodd} and \eqref{eq:modifiedkarlinmcgregorheven}, we can recursively compute the recurrence coefficients $(c_n)_{n\in\mathbb{N}}$ which correspond to the modified polynomials $(P_n(x))_{n\in\mathbb{N}_0}$. Alternatively, one can compute $(c_n)_{n\in\mathbb{N}}$ from \eqref{eq:karlinmcgregorgamma1even} and \eqref{eq:karlinmcgregorgamma1odd} because the recurrence coefficients are linked to those belonging to $(K_n^{(\alpha,\beta)}(x))_{n\in\mathbb{N}_0}$ by multiplication with $K_{n-1}^{(\alpha,\beta)}(\gamma_1)/(\gamma_1K_n^{(\alpha,\beta)}(\gamma_1))$. We obtain
\begin{align*}
c_{2n-1}&=\frac{\sqrt{\beta-1}}{\sqrt{\alpha-1}+\sqrt{\beta-1}}\\
&\quad\times\left[1-\sqrt{\alpha-1}\cdot\frac{\sqrt{\alpha-1}\sqrt{\beta-1}-1}{((\alpha-2)\sqrt{\beta-1}+(\beta-2)\sqrt{\alpha-1})\cdot n+\sqrt{\alpha-1}+\sqrt{\beta-1}}\right],\\
c_{2n}&=\frac{\sqrt{\alpha-1}}{\sqrt{\alpha-1}+\sqrt{\beta-1}}\left[1-\sqrt{\beta-1}\cdot\frac{\sqrt{\alpha-1}\sqrt{\beta-1}-1}{((\alpha-2)\sqrt{\beta-1}+(\beta-2)\sqrt{\alpha-1})\cdot n+\beta\sqrt{\alpha-1}}\right]
\end{align*}
and particularly
\begin{align*}
\lim_{n\to\infty}c_{2n-1}&=\frac{\sqrt{\beta-1}}{\sqrt{\alpha-1}+\sqrt{\beta-1}},\\
\lim_{n\to\infty}c_{2n}&=\frac{\sqrt{\alpha-1}}{\sqrt{\alpha-1}+\sqrt{\beta-1}}.
\end{align*}
Therefore, both $(c_{2n-1})_{n\in\mathbb{N}}$ and $(c_{2n})_{n\in\mathbb{N}}$ are nondecreasing but $(c_n)_{n\in\mathbb{N}}$ is not bounded from above by $1/2$ unless $\alpha=\beta$. In particular, Theorem~\ref{thm:szwarcnonnegII} cannot be applied as soon as $\alpha\neq\beta$.\\

For every $n\in\mathbb{N}$, we compute
\begin{equation*}
\alpha_n=\begin{cases} \frac{\sqrt{\beta}}{\sqrt{\alpha-1}+\sqrt{\beta-1}}, & n=1, \\ \frac{\sqrt{\alpha-1}}{\sqrt{\alpha-1}+\sqrt{\beta-1}}, & n\;\mbox{even}, \\ \frac{\sqrt{\beta-1}}{\sqrt{\alpha-1}+\sqrt{\beta-1}}, & \mbox{else}, \end{cases}
\end{equation*}
so the coefficients in the orthonormal/monic normalization become periodic. This shows that $(P_n(x))_{n\in\mathbb{N}_0}$ belongs to the class of Geronimus polynomials \cite{Le00} and that nonnegative linearization of products also follows directly from a general criterion in \cite{Sz94b} (without using nonnegative linearization of products for the Karlin--McGregor polynomials): if $\alpha\leq\beta$, then \cite[Theorem 3 (i)]{Sz94b} can be applied, and if $\alpha>\beta$, then \cite[Theorem 3 (ii) combined with Remark 3]{Sz94b} works. For the special case $\alpha=\beta$, nonnegative linearization of products also follows from \cite[3 (g) (i)]{La83} (via explicit formulas for the linearization coefficients of the corresponding Geronimus polynomials).\\

Moreover, our construction based on the Karlin--McGregor polynomials makes it possible to compute the linearization coefficients $g(m,n;k)$ explicitly because the linearization coefficients belonging to the Karlin--McGregor polynomials are explicitly known \cite[Theorem 4.1]{Eh94} and just have to be multiplied by $K_k^{(\alpha,\beta)}(\gamma_1)/(K_m^{(\alpha,\beta)}(\gamma_1)K_n^{(\alpha,\beta)}(\gamma_1))$ (which is explicitly known due to \eqref{eq:karlinmcgregorgamma1even} and \eqref{eq:karlinmcgregorgamma1odd}). We refrain from stating the corresponding formulas, however.\\

Coming back to our necessary condition \eqref{eq:crucialnecessary} on the dual space $\widehat{\mathbb{N}_0}$ from the beginning of the section, we see from \eqref{eq:measuremodifiedKarlinMcGregorabscont} to \eqref{eq:measuremodifiedKarlinMcGregordisc} that the modified polynomials $(P_n(x))_{n\in\mathbb{N}_0}$ satisfy this condition if and only if $\alpha<\beta$. We now also come back to the problem ``$h(1)<2$'' and first observe that if one is not interested in the full function $h$ but only in $h(1)$, then one can also argue in the following way which is just based on the first recurrence coefficient $c_1$ and less computational than the proof of the full equations \eqref{eq:modifiedkarlinmcgregorhodd}, \eqref{eq:modifiedkarlinmcgregorheven}: $c_1$ satisfies
\begin{equation*}
x^2=(1-c_1)P_2(x)+c_1,
\end{equation*}
so
\begin{equation*}
h(1)=\frac{1}{c_1}=1-\frac{1}{P_2(0)}=1-\frac{K_2^{(\alpha,\beta)}(\gamma_1)}{K_2^{(\alpha,\beta)}(0)}.
\end{equation*}
Since
\begin{equation*}
K_2^{(\alpha,\beta)}(x)=\frac{\alpha x^2-1}{\alpha-1},
\end{equation*}
we obtain
\begin{equation*}
h(1)=\alpha\gamma_1^2=\frac{1}{\beta}(\sqrt{\alpha-1}+\sqrt{\beta-1})^2
\end{equation*}
(cf. \eqref{eq:modifiedkarlinmcgregorhodd}). In particular, one has $h(1)<2$ if and only if $\alpha<3\beta-2\sqrt{2\beta^2-2\beta}$. Since $[2,3\beta-2\sqrt{2\beta^2-2\beta})$ is a (proper) subset of $[2,\beta)$, the corresponding measures are absolutely continuous with density given by the first case of \eqref{eq:measuremodifiedKarlinMcGregorabscont}. Moreover, we obtain that the necessary condition \eqref{eq:crucialnecessary} is not sufficient for $h(1)<2$.

\begin{figure}
\centering
\includegraphics[width=0.60\textwidth]{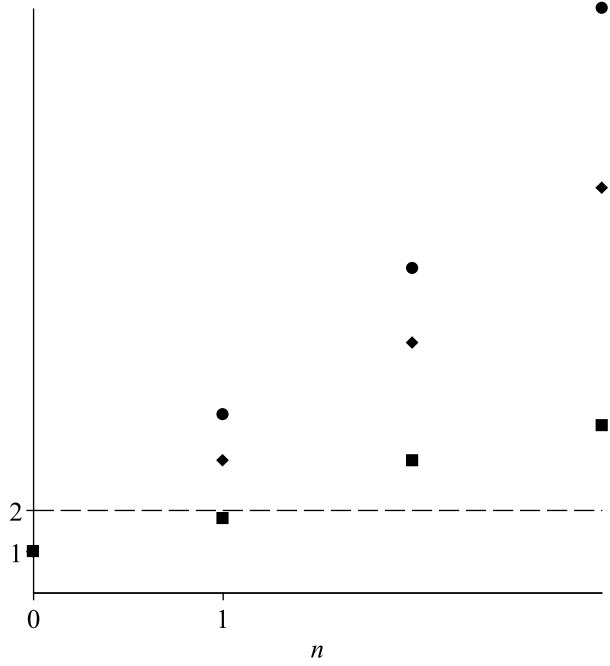}
\caption{The first Haar weights $h(n)$ belonging to $(P_n(x))_{n\in\mathbb{N}_0}$, cf. \eqref{eq:modifiedkarlinmcgregorhodd} and \eqref{eq:modifiedkarlinmcgregorheven}, for $(\alpha,\beta)=(2,5)$ (box), $(\alpha,\beta)=(5,5)$ (diamond) and $(\alpha,\beta)=(8,5)$ (circle). In the first case, one has $h(1)<2$.}\label{fig:hmeasures}
\end{figure}

\begin{theorem}\label{thm:smallerthanoneplusepsilon}
Let $\alpha,\beta\geq2$, and let $P_n(x)=K_n^{(\alpha,\beta)}(\gamma_1x)/K_n^{(\alpha,\beta)}(\gamma_1)\;(n\in\mathbb{N}_0)$.
\begin{enumerate}[(i)]
\item For every $\epsilon\in(0,1)$, there exists a polynomial hypergroup on $\mathbb{N}_0$ such that $h(1)=1+\epsilon$. More precisely, for any choice of $\alpha$ the parameter $\beta$ can be chosen in such a way that the hypergroup induced by the sequence $(P_n(x))_{n\in\mathbb{N}_0}$ has the desired property; for this example type, the dual space $\widehat{\mathbb{N}_0}$ is of the form $[-1,-1+\delta]\cup[1-\delta,1]$ with $\delta\in(0,1-\sqrt{(1-\epsilon)/(1+\epsilon)}]$. Furthermore, for $\alpha=2$ and $\beta=(2+2\sqrt{1-\epsilon^2})/\epsilon^2$ one additionally has that $\widehat{\mathbb{N}_0}$ equals the maximal possible set $[-1,-\sqrt{(1-\epsilon)/(1+\epsilon)}]\cup[\sqrt{(1-\epsilon)/(1+\epsilon)},1]$ (cf. Proposition~\ref{prp:dualspace} and Remark~\ref{rem:dualspace}).
\item For any choice of $\alpha,\beta$, the polynomial hypergroup induced by $(P_n(x))_{n\in\mathbb{N}_0}$ satisfies $h(n)\geq2$ for all $n\geq2$. Moreover, $h$ is nondecreasing and of quadratic growth.
\end{enumerate}
\end{theorem}

\begin{proof}
\begin{enumerate}[(i)]
\item Let $\epsilon\in(0,1)$, and let $\alpha\geq2$ be arbitrary. Then, by the preceding calculations,
\begin{equation*}
h(1)=\frac{1}{\beta}(\sqrt{\alpha-1}+\sqrt{\beta-1})^2\to1\;(\beta\to\infty).
\end{equation*}
This yields that $\beta$ can be chosen such that $h(1)=1+\epsilon$. By the preceding considerations and Proposition~\ref{prp:dualspace}, $\widehat{\mathbb{N}_0}=\mathrm{supp}\;\mu=[-1,-1+\delta]\cup[1-\delta,1]$ with $\delta\in(0,1-\sqrt{(1-\epsilon)/(1+\epsilon)}]$.\\

Now let $\alpha=2$ and $\beta=(2+2\sqrt{1-\epsilon^2})/\epsilon^2$. Then $h(1)=1+\epsilon$ and
\begin{equation*}
\widehat{\mathbb{N}_0}=\mathrm{supp}\;\mu=\left[-1,-\frac{\gamma_2}{\gamma_1}\right]\cup\left[\frac{\gamma_2}{\gamma_1},1\right]=\left[-1,-\sqrt{\frac{1-\epsilon}{1+\epsilon}}\right]\cup\left[\sqrt{\frac{1-\epsilon}{1+\epsilon}},1\right].
\end{equation*}
\item For every $n\in\mathbb{N}$, the explicit formulas \eqref{eq:modifiedkarlinmcgregorhodd} and \eqref{eq:modifiedkarlinmcgregorheven} for $h$ yield
\begin{align*}
\frac{h(2n)}{h(2n-1)}&=\left[1+\sqrt{\beta-1}\cdot\frac{\sqrt{\alpha-1}\sqrt{\beta-1}-1}{((\alpha-2)\sqrt{\beta-1}+(\beta-2)\sqrt{\alpha-1})\cdot n+\sqrt{\alpha-1}+\sqrt{\beta-1}}\right]^2\\
&\geq1
\end{align*}
and
\begin{equation*}
\frac{h(2n+1)}{h(2n)}=\left[1+\sqrt{\alpha-1}\cdot\frac{\sqrt{\alpha-1}\sqrt{\beta-1}-1}{((\alpha-2)\sqrt{\beta-1}+(\beta-2)\sqrt{\alpha-1})\cdot n+\beta\sqrt{\alpha-1}}\right]^2\geq1
\end{equation*}
for every $n\in\mathbb{N}$, which shows that $h$ is nondecreasing. Hence, by \eqref{eq:modifiedkarlinmcgregorheven} we have
\begin{align*}
h(n)&\geq h(2)\\
&=\frac{1}{\beta}\left[\frac{\alpha-2}{\sqrt{\alpha-1}}+\frac{\beta-2}{\sqrt{\beta-1}}+\frac{\beta}{\sqrt{\beta-1}}\right]^2\\
&\geq\frac{1}{\beta}\left[\frac{\beta-2}{\sqrt{\beta-1}}+\frac{\beta}{\sqrt{\beta-1}}\right]^2\\
&=4\cdot\frac{\beta-1}{\beta}\\
&\geq2
\end{align*}
for all $n\geq2$. We have already observed that $h$ is of quadratic growth.
\end{enumerate}
\end{proof}

\begin{corollary}\label{cor:smallerthanoneplusepsilon}
The converse of Theorem~\ref{thm:greaterthantwo} and the converses of Corollary~\ref{cor:greaterthantwoconcoeff} to Corollary~\ref{cor:nevai} are not true.
\end{corollary}

\begin{proof}
Let $\beta\geq2$ and $\alpha\geq3\beta-2\sqrt{2\beta^2-2\beta}$ with $\alpha\neq\beta$, and let $P_n(x)=K_n^{(\alpha,\beta)}(\gamma_1x)/K_n^{(\alpha,\beta)}(\gamma_1)\;(n\in\mathbb{N}_0)$. Then, as consequence of the second part of Theorem~\ref{thm:smallerthanoneplusepsilon} and the preceding notes, we have $h(n)\geq2$ for all $n\in\mathbb{N}$ but $\widehat{\mathbb{N}_0}\neq[-1,1]$ (in fact, for $\alpha\in[3\beta-2\sqrt{2\beta^2-2\beta},\beta)$ we do not even have $0\in\widehat{\mathbb{N}_0}$). This shows that the converse of Theorem~\ref{thm:greaterthantwo} is not true. The latter implies that the converses of Corollary~\ref{cor:greaterthantwoconcoeff} to Corollary~\ref{cor:nevai} are also not true.
\end{proof}

We briefly consider the special case $\alpha=2$. As a consequence of \eqref{eq:kmspecialodd} and \eqref{eq:kmspecialeven}, in this special case we have
\begin{align}
\label{eq:kmspecialoddmod} P_{2n-1}(x)&=\frac{1}{\sqrt{\beta-1}n-(n-1)}x\\
\notag &\quad\times\left(\sqrt{\beta-1}U_{n-1}\left(\frac{(2\sqrt{\beta-1}+\beta)x^2-\beta}{2\sqrt{\beta-1}}\right)-U_{n-2}\left(\frac{(2\sqrt{\beta-1}+\beta)x^2-\beta}{2\sqrt{\beta-1}}\right)\right)
\end{align}
and
\begin{align}
\label{eq:kmspecialevenmod} P_{2n}(x)&=\frac{1}{(\beta-2)n+\beta}\\
\notag &\quad\times\left((\beta-1)U_n\left(\frac{(2\sqrt{\beta-1}+\beta)x^2-\beta}{2\sqrt{\beta-1}}\right)-U_{n-2}\left(\frac{(2\sqrt{\beta-1}+\beta)x^2-\beta}{2\sqrt{\beta-1}}\right)\right)
\end{align}
for every $n\in\mathbb{N}$. Choosing $(\alpha,\beta)=(2,5)$ in Theorem~\ref{thm:smallerthantwo} was motivated by getting ``nice values'' for $\sqrt{\alpha-1}$ and $\sqrt{\beta-1}$. In particular, this special choice allowed us to obtain very simple explicit expressions for the sequence $(c_n)_{n\in\mathbb{N}}$. Due to \eqref{eq:measuremodifiedKarlinMcGregorabscont}, the density takes the simple form
\begin{equation*}
\mathrm{d}\mu(x)=\frac{15\sqrt{(1-x^2)(9x^2-1)}}{2\pi|x|(10-9x^2)}\chi_{\left(-1,-\frac{1}{3}\right)\cup\left(\frac{1}{3},1\right)}(x)\,\mathrm{d}x
\end{equation*}
in this case (cf. the solid line in Figure~\ref{fig:orthmeasures}), so the dual space is given by
\begin{equation*}
\widehat{\mathbb{N}_0}=\mathrm{supp}\;\mu=\left[-1,-\frac{1}{3}\right]\cup\left[\frac{1}{3},1\right],
\end{equation*}
which is the maximal possible dual space $\widehat{\mathbb{N}_0}$ for a polynomial hypergroup with $h(1)=9/5$ (cf. Proposition~\ref{prp:dualspace} and Remark~\ref{rem:dualspace}). Furthermore, due to \eqref{eq:kmspecialoddmod} and \eqref{eq:kmspecialevenmod} the polynomials take the simple form
\begin{align*}
P_{2n-1}(x)&=\frac{1}{n+1}x\left(2U_{n-1}\left(\frac{9x^2-5}{4}\right)-U_{n-2}\left(\frac{9x^2-5}{4}\right)\right),\\
P_{2n}(x)&=\frac{1}{3n+5}\left(4U_n\left(\frac{9x^2-5}{4}\right)-U_{n-2}\left(\frac{9x^2-5}{4}\right)\right)
\end{align*}
for all $n\in\mathbb{N}$.\\

We finally construct another type of polynomial hypergroups with $h(1)<2$ (and even ``$h(1)=1+\epsilon$''). It does not rely on the Karlin--McGregor polynomials; the dual space $\widehat{\mathbb{N}_0}$ is discrete, and $h$ is of exponential growth now.

\begin{theorem}\label{thm:smallerthanoneplusepsilon2}
Let $(\lambda_n)_{n\in\mathbb{N}_0}\subseteq(0,1)$ with $\lim_{n\to\infty}\lambda_{2n}=1$ satisfy both $\lambda_{2n-2}+\lambda_{2n-1}\leq\lambda_{2n}$ and $\lambda_{2n-1}+\lambda_{2n}\leq\lambda_{2n+2}$ for every $n\in\mathbb{N}$, and let $(Q_n(x))_{n\in\mathbb{N}_0}\subseteq\mathbb{R}[x]$ be given by the recurrence relation $Q_0(x)=1$, $Q_1(x)=x/\lambda_0$,
\begin{equation*}
x Q_n(x)=\lambda_n Q_{n+1}(x)+\lambda_{n-1}Q_{n-1}(x)\;(n\in\mathbb{N}).
\end{equation*}
The following hold:
\begin{enumerate}[(i)]
\item The sequence $(Q_n(1))_{n\in\mathbb{N}_0}$ is strictly positive and strictly increasing.
\item The sequence $(P_n(x))_{n\in\mathbb{N}_0}$ defined by $P_n(x):=Q_n(x)/Q_n(1)\;(n\in\mathbb{N}_0)$ satisfies nonnegative linearization of products, and $(Q_n(x))_{n\in\mathbb{N}_0}$ are the orthonormal polynomials which correspond to $(P_n(x))_{n\in\mathbb{N}_0}$.
\item The dual space $\widehat{\mathbb{N}_0}$ satisfies
\begin{equation*}
\widehat{\mathbb{N}_0}=\mathrm{supp}\;\mu=\{\pm1\}\cup\{\pm x_n:n\in\mathbb{N}\}
\end{equation*}
with a strictly increasing sequence $(x_n)_{n\in\mathbb{N}}\subseteq(0,1)$ with $\lim_{n\to\infty}x_n=1$.
\item For every $\epsilon\in(0,1)$, the sequence $(\lambda_n)_{n\in\mathbb{N}_0}$ can be chosen in such a way that the polynomial hypergroup induced by $(P_n(x))_{n\in\mathbb{N}_0}$ fulfills $h(1)=1+\epsilon$; in that case, one has $(x_n)_{n\in\mathbb{N}}\subseteq[\sqrt{(1-\epsilon)/(1+\epsilon)},1)$. An explicit construction is as follows: let $(s_n)_{n\in\mathbb{N}_0}\subseteq(0,1)$ be any null sequence which is convex (i.e., $s_{n+1}\leq(s_n+s_{n+2})/2$ for all $n\in\mathbb{N}_0$). Then the sequence $(\lambda_n)_{n\in\mathbb{N}_0}\subseteq\mathbb{R}$ given by
\begin{equation*}
\lambda_n:=\begin{cases} 1-s_{\frac{n}{2}}, & n\;\mbox{even}, \\ s_{\frac{n+1}{2}}-s_{\frac{n+3}{2}}, & n\;\mbox{odd} \end{cases}
\end{equation*}
satisfies the conditions above, and if $s_0=1-1/\sqrt{1+\epsilon}$, then $h(1)=1+\epsilon$.\footnote{Our construction is motivated by similar (but less general) ideas in \cite[Remark 1 p. 427]{MS01}.}
\item For any choice of $(\lambda_n)_{n\in\mathbb{N}_0}$, the polynomial hypergroup induced by $(P_n(x))_{n\in\mathbb{N}_0}$ satisfies $h(n)>4$ for all $n\geq2$. Moreover, $h$ is strictly increasing and of exponential growth.
\end{enumerate}
\end{theorem}

\begin{proof}
\begin{enumerate}[(i)]
\item For every $n\in\mathbb{N}$, we compute
\begin{equation*}
Q_{n+1}(1)=Q_n(1)+\underbrace{\frac{1-\lambda_{n-1}-\lambda_n}{\lambda_n}}_{>0}Q_n(1)+\underbrace{\frac{\lambda_{n-1}}{\lambda_n}}_{>0}(Q_n(1)-Q_{n-1}(1)).
\end{equation*}
Since $Q_0(1)=1$ and $Q_1(1)=1/\lambda_0>1$, this yields the assertion.
\item As a consequence of (i), $(P_n(x))_{n\in\mathbb{N}_0}$ is well-defined. By \cite[Corollary 2 (ii)]{MS01}, $(Q_n(x))_{n\in\mathbb{N}_0}$ satisfies nonnegative linearization of products. Hence, (i) implies that $(P_n(x))_{n\in\mathbb{N}_0}$ satisfies nonnegative linearization of products, too. It is clear from the recurrence relations that $(Q_n(x))_{n\in\mathbb{N}_0}$ are the orthonormal polynomials which correspond to $(P_n(x))_{n\in\mathbb{N}_0}$.
\item As a consequence of \cite[Remark 2 p. 427]{MS01}, the difference operator $L:\ell^2\rightarrow\ell^2$,
\begin{equation*}
(L u)_n:=\begin{cases} \lambda_0u_1, & n=0, \\ \lambda_n u_{n+1}+\lambda_{n-1}u_{n-1}, & n\in\mathbb{N} \end{cases}
\end{equation*}
has norm $1$, the operator $L^2-\id$ is compact and there exists a strictly increasing sequence $(x_n)_{n\in\mathbb{N}}\subseteq[0,1)$ with $\lim_{n\to\infty}x_n=1$ and
\begin{equation*}
\mathrm{supp}\;\mu=\sigma(L)=\{\pm1\}\cup\{\pm x_n:n\in\mathbb{N}\}.
\end{equation*}
It remains to show that $\widehat{\mathbb{N}_0}=\mathrm{supp}\;\mu$ and $0\notin\widehat{\mathbb{N}_0}$. This can be seen as follows: let $(R_n(x))_{n\in\mathbb{N}_0}\subseteq\mathbb{R}[x]$ be defined by $R_n(x^2)=P_{2n}(x)$ (this approach is motivated by \cite[Section 6]{MS01}). Then $(R_n(x))_{n\in\mathbb{N}_0}$ satisfies the recurrence relation $R_0(x)=1$, $R_1(x)=(x-c_1)/a_1$,
\begin{equation*}
R_1(x)R_n(x)=a_n^R R_{n+1}(x)+b_n^R R_n(x)+c_n^R R_{n-1}(x)\;(n\in\mathbb{N}),
\end{equation*}
where
\begin{align*}
a_n^R&:=\frac{a_{2n}a_{2n+1}}{a_1}\in(0,1),\\
b_n^R&:=\frac{a_{2n}c_{2n+1}+c_{2n}a_{2n-1}-c_1}{a_1}\in(0,1),\\
c_n^R&:=\frac{c_{2n}c_{2n-1}}{a_1}\in(0,1)
\end{align*}
(as usual, $(c_n)_{n\in\mathbb{N}}\subseteq(0,1)$ and $a_n\equiv1-c_n$ shall denote the recurrence coefficients which belong to the sequence $(P_n(x))_{n\in\mathbb{N}_0}$). The estimation $b_n^R>0$ follows from the monotonicity behavior of $(\lambda_{2n})_{n\in\mathbb{N}_0}$ because
\begin{equation*}
a_{2n}c_{2n+1}+c_{2n}a_{2n-1}=\lambda_{2n}^2+\lambda_{2n-1}^2>\lambda_{2n}^2>\lambda_0^2=c_1
\end{equation*}
for every $n\in\mathbb{N}$; the remaining estimations are obvious from $a_n^R+b_n^R+c_n^R=1\;(n\in\mathbb{N})$. Since, for every $n\geq2$,
\begin{equation*}
1>c_{2n-1}=\frac{\lambda_{2n-2}^2}{a_{2n-2}}>\lambda_{2n-2}^2\to1\;(n\to\infty),
\end{equation*}
we have $\lim_{n\to\infty}c_{2n-1}=1$ (hence $\lim_{n\to\infty}a_{2n-1}=0$) and consequently
\begin{equation*}
c_{2n}=1-\frac{\lambda_{2n}^2}{c_{2n+1}}\to0\;(n\to\infty)
\end{equation*}
(hence $\lim_{n\to\infty}a_{2n}=1$). Therefore, we have
\begin{align*}
\lim_{n\to\infty}a_n^R&=0,\\
\lim_{n\to\infty}b_n^R&=1,\\
\lim_{n\to\infty}c_n^R&=0.
\end{align*}
Let $\mu_R$ denote the orthogonalization (probability) measure of $(R_n(x))_{n\in\mathbb{N}_0}$. $\mu_R$ can be regarded as pushforward measure of $\mu$. It is clear from the construction that
\begin{equation*}
\mathrm{supp}\;\mu_R=\{1\}\cup\{x_n^2:n\in\mathbb{N}\}.
\end{equation*}
By \cite[Proposition 4]{FLS05}, the behavior of the sequences $(a_n^R)_{n\in\mathbb{N}}$, $(b_n^R)_{n\in\mathbb{N}}$ and $(c_n^R)_{n\in\mathbb{N}}$ as obtained above implies that $(R_n(x))_{n\in\mathbb{N}_0}$ induces a polynomial hypergroup of `strong compact type' on $\mathbb{N}_0$, which means that the operator $T_n-\id$ is compact on $\ell^1(h)$ for every $n\in\mathbb{N}_0$.\footnote{Note that the sequence $(R_n(x))_{n\in\mathbb{N}_0}$ is not symmetric. The definitions of a polynomial hypergroup, the translation operator $T_n$, the Haar function $h$ and the space $\ell^1(h)$ are the same as for the symmetric case (as recalled in Section~\ref{sec:intro}), however.} By \cite[Theorem 2]{FLS05}, this yields
\begin{equation*}
\left\{x\in\mathbb{R}:\max_{n\in\mathbb{N}_0}|R_n(x)|=1\right\}=\mathrm{supp}\;\mu_R.
\end{equation*}
Therefore, we obtain that
\begin{align*}
\widehat{\mathbb{N}_0}&\subseteq\left\{x\in\mathbb{R}:\max_{n\in\mathbb{N}_0}|P_{2n}(x)|=1\right\}\\
&=\left\{x\in\mathbb{R}:\max_{n\in\mathbb{N}_0}|R_n(x^2)|=1\right\}\\
&=\{\pm1\}\cup\{\pm x_n:n\in\mathbb{N}\}.
\end{align*}
Since, however,
\begin{equation*}
\{\pm1\}\cup\{\pm x_n:n\in\mathbb{N}\}=\mathrm{supp}\;\mu\subseteq\widehat{\mathbb{N}_0},
\end{equation*}
we obtain equality. Moreover, we have $0\notin\widehat{\mathbb{N}_0}$ because
\begin{equation*}
|P_{2n}(0)|=\prod_{k=0}^{n-1}\frac{c_{2k+1}}{a_{2k+1}}\to\infty\;(n\to\infty)
\end{equation*}
by the limiting behavior of $(c_{2n-1})_{n\in\mathbb{N}}$ and $(a_{2n-1})_{n\in\mathbb{N}}$.
\item We first show that our explicit construction works. Since the convexity of the null sequence $(s_n)_{n\in\mathbb{N}_0}\subseteq(0,1)$ implies that $(s_n)_{n\in\mathbb{N}_0}$ is strictly decreasing, we have $(\lambda_n)_{n\in\mathbb{N}_0}\subseteq(0,1)$. Moreover, it is clear that $\lim_{n\to\infty}\lambda_{2n}=1$. Finally, by convexity we have
\begin{equation*}
\lambda_0+\lambda_1=1-s_0+s_1-s_2\leq1-s_1=\lambda_2
\end{equation*}
and, for every $n\in\mathbb{N}$,
\begin{equation*}
\lambda_{2n+1}=s_{n+1}-s_{n+2}\leq s_{n}-s_{n+1}=\lambda_{2n-1}
\end{equation*}
and
\begin{equation*}
\lambda_{2n-1}+\lambda_{2n}=s_n-s_{n+1}+1-s_n=1-s_{n+1}=\lambda_{2n+2}.
\end{equation*}
If $s_0=1-1/\sqrt{1+\epsilon}$, then (ii) yields
\begin{equation*}
h(1)=Q_1^2(1)=\frac{1}{\lambda_0^2}=\frac{1}{(1-s_0)^2}=1+\epsilon.
\end{equation*}
It remains to show that $h(1)=1+\epsilon$ implies that $(x_n)_{n\in\mathbb{N}}\subseteq[\sqrt{(1-\epsilon)/(1+\epsilon)},1)$. This is an immediate consequence of Proposition~\ref{prp:dualspace}, however.
\item As a consequence of (i) and (ii), the sequence $(h(n))_{n\in\mathbb{N}_0}$ coincides with $(Q_n^2(1))_{n\in\mathbb{N}_0}$ and is therefore strictly increasing. Hence, we have
\begin{equation*}
h(n)\geq h(2)=Q_2^2(1)=\left(\frac{1-\lambda_0^2}{\lambda_0\lambda_1}\right)^2>\left(\frac{1-\lambda_0^2}{\lambda_0(1-\lambda_0)}\right)^2=\left(1+\frac{1}{\lambda_0}\right)^2>4
\end{equation*}
for all $n\geq2$. It remains to prove that $h$ is of exponential growth. Let $n\in\mathbb{N}$. Then
\begin{equation*}
x Q_{2n}(x)=\lambda_{2n}Q_{2n+1}(x)+\lambda_{2n-1}Q_{2n-1}(x)
\end{equation*}
and consequently
\begin{align*}
x^2Q_{2n}(x)&=\lambda_{2n}x Q_{2n+1}(x)+\lambda_{2n-1}x Q_{2n-1}(x)\\
&=\lambda_{2n}\lambda_{2n+1}Q_{2n+2}(x)+(\lambda_{2n}^2+\lambda_{2n-1}^2)Q_{2n}(x)+\lambda_{2n-1}\lambda_{2n-2}Q_{2n-2}(x),
\end{align*}
so
\begin{equation*}
Q_{2n}(1)=\lambda_{2n}\lambda_{2n+1}Q_{2n+2}(1)+(\lambda_{2n}^2+\lambda_{2n-1}^2)Q_{2n}(1)+\lambda_{2n-1}\lambda_{2n-2}Q_{2n-2}(1).
\end{equation*}
The latter yields
\begin{equation*}
Q_{2n+2}(1)=\frac{1-\lambda_{2n}^2-\lambda_{2n-1}^2}{\lambda_{2n}\lambda_{2n+1}}Q_{2n}(1)-\frac{\lambda_{2n-1}\lambda_{2n-2}}{\lambda_{2n}\lambda_{2n+1}}Q_{2n-2}(1).
\end{equation*}
Since $Q_{2n}(1)>Q_{2n-2}(1)$ by (i), we get
\begin{align*}
Q_{2n+2}(1)&>\frac{1-\lambda_{2n}^2-\lambda_{2n-1}^2-\lambda_{2n-1}\lambda_{2n-2}}{\lambda_{2n}\lambda_{2n+1}}Q_{2n}(1)\\
&=\frac{1-\lambda_{2n}^2-\lambda_{2n-1}(\lambda_{2n-1}+\lambda_{2n-2})}{\lambda_{2n}\lambda_{2n+1}}Q_{2n}(1)\\
&\geq\frac{1-\lambda_{2n}^2-\lambda_{2n-1}\lambda_{2n}}{\lambda_{2n}\lambda_{2n+1}}Q_{2n}(1)\\
&=\frac{1-\lambda_{2n}(\lambda_{2n}+\lambda_{2n-1})}{\lambda_{2n}\lambda_{2n+1}}Q_{2n}(1)\\
&\geq\frac{1-\lambda_{2n}\lambda_{2n+2}}{\lambda_{2n}\lambda_{2n+1}}Q_{2n}(1)\\
&>\frac{1-\lambda_{2n}}{\lambda_{2n+1}}Q_{2n}(1)\\
&\geq\frac{1+\lambda_{2n+1}-\lambda_{2n+2}}{\lambda_{2n+1}}Q_{2n}(1)\\
&\geq\frac{1+\lambda_{2n+1}+\lambda_{2n+1}-\lambda_{2n+4}}{\lambda_{2n+1}}Q_{2n}(1)\\
&>2Q_{2n}(1).
\end{align*}
This shows that $(Q_{2n}(1))_{n\in\mathbb{N}_0}$ is of exponential growth. Therefore, we obtain that $(h(2n))_{n\in\mathbb{N}_0}=(Q_{2n}^2(1))_{n\in\mathbb{N}_0}$ (and hence $h$) is of exponential growth.
\end{enumerate}
\end{proof}

\begin{remark}
Reconsider the explicit construction studied in the proof of Theorem~\ref{thm:smallerthanoneplusepsilon2} (iv), now for $s_0\geq1-\sqrt{2}/2$. In this case, one has $h(1)\geq2$ (and consequently $h(n)\geq2$ for all $n\in\mathbb{N}$ by Theorem~\ref{thm:smallerthanoneplusepsilon2} (v)). However, the dual space $\widehat{\mathbb{N}_0}$ is a discrete subset of $[-1,1]$ by Theorem~\ref{thm:smallerthanoneplusepsilon2} (iii). Therefore, the example provides an alternative proof of Corollary~\ref{cor:smallerthanoneplusepsilon}.
\end{remark}

\section{Additional notes on the dual spaces}\label{sec:dual}

Recall that $\widehat{\mathbb{N}_0}$ can be identified with the \textit{Hermitian} structure space $\Delta_s(\ell^1(h))$. In contrast to Abelian groups, there can occur nonsymmetric characters or, in other words, $\Delta_s(\ell^1(h))$ may differ from the (full) structure space $\Delta(\ell^1(h))$. Let
\begin{equation*}
\mathcal{X}^b(\mathbb{N}_0):=\left\{z\in\mathbb{C}:\max_{n\in\mathbb{N}_0}|P_n(z)|=1\right\}.
\end{equation*}
Via the homeomorphism $\mathcal{X}^b(\mathbb{N}_0)\rightarrow\Delta(\ell^1(h))$, $z\mapsto\varphi_z$ with
\begin{equation*}
\varphi_z(f):=\sum_{k=0}^\infty f(k)\overline{P_k(z)}h(k)\;(f\in\ell^1(h)),
\end{equation*}
the (compact) set $\mathcal{X}^b(\mathbb{N}_0)$ can be identified with $\Delta(\ell^1(h))$ \cite{La83,La05}; the occurring element of $\ell^{\infty}$ which maps $k$ to $P_k(z)$ is called the character associated with $z\in\mathcal{X}^b(\mathbb{N}_0)$, and $\widehat{\mathbb{N}_0}$ is just the intersection of $\mathcal{X}^b(\mathbb{N}_0)$ with $\mathbb{R}$.\\

If $h$ is of subexponential growth, then the three dual objects $\mathrm{supp}\;\mu$, $\widehat{\mathbb{N}_0}$ and $\mathcal{X}^b(\mathbb{N}_0)$ coincide \cite{Hu70,Sz95,Vo87,Vo88}. Concerning the polynomial hypergroups considered in this paper, there are examples where $\mathrm{supp}\;\mu$, $\widehat{\mathbb{N}_0}$ and $\mathcal{X}^b(\mathbb{N}_0)$ are pairwise distinct (cosh-polynomials \cite[Sect. 6]{La05}, Karlin--McGregor polynomials for $(\alpha,\beta)\neq(2,2)$ \cite[Sect. 4]{Pe11}).\\

We now show that $\widehat{\mathbb{N}_0}=\mathcal{X}^b(\mathbb{N}_0)$ for our two example types with ``$h(1)=1+\epsilon$'' considered in Section~\ref{sec:karlinmcgregor} (recall that the analogous equality $\mathrm{supp}\;\mu=\widehat{\mathbb{N}_0}$ has already been established in Section~\ref{sec:karlinmcgregor} for both of these example types).\\

If $(P_n(x))_{n\in\mathbb{N}_0}$ is as in Theorem~\ref{thm:smallerthanoneplusepsilon}, then $\widehat{\mathbb{N}_0}=\mathcal{X}^b(\mathbb{N}_0)$ is immediate from the fact that $h$ is of quadratic (hence subexponential) growth.\\

The case $(P_n(x))_{n\in\mathbb{N}_0}$ as in Theorem~\ref{thm:smallerthanoneplusepsilon2} is more interesting because here $h$ is of exponential growth. It can be obtained from the relationship to hypergroups of strong compact type which was already used in the proof of Theorem~\ref{thm:smallerthanoneplusepsilon2}: let $(R_n(x))_{n\in\mathbb{N}_0}$, $\mu_R$ and $(x_n)_{n\in\mathbb{N}}\subseteq(0,1)$ be as in the proof of Theorem~\ref{thm:smallerthanoneplusepsilon2} (iii). Then \cite[Theorem 2]{FLS05} yields
\begin{equation*}
\left\{z\in\mathbb{C}:\max_{n\in\mathbb{N}_0}|R_n(z)|=1\right\}=\mathrm{supp}\;\mu_R=\{1\}\cup\{x_n^2:n\in\mathbb{N}\}.
\end{equation*}
Hence, we get
\begin{align*}
\mathcal{X}^b(\mathbb{N}_0)&\subseteq\left\{z\in\mathbb{C}:\max_{n\in\mathbb{N}_0}|P_{2n}(z)|=1\right\}\\
&=\left\{z\in\mathbb{C}:\max_{n\in\mathbb{N}_0}|R_n(z^2)|=1\right\}\\
&=\{\pm1\}\cup\{\pm x_n:n\in\mathbb{N}\}.
\end{align*}
In particular, $\mathcal{X}^b(\mathbb{N}_0)$ is a subset of $\mathbb{R}$, which implies the assertion.

\section{Open problems}\label{sec:openproblems}

We finish our paper with a collection of some open problems:
\begin{enumerate}[(i)]
\item Is $h(2)\geq2$ always true?
\item Is $\liminf_{n\to\infty}h(n)\geq2$ always true?
\item Is $h(n)\geq2\;(n\in\mathbb{N}\backslash\{1\})$ always true?
\item Is $h(1)\geq2$ sufficient for $h(2)\geq2$, $\liminf_{n\to\infty}h(n)\geq2$ or $h(n)\geq2\;(n\in\mathbb{N})$?
\item Is $0\in\widehat{\mathbb{N}_0}$ sufficient for $h(2)\geq2$, $\liminf_{n\to\infty}h(n)\geq2$ or $h(n)\geq2\;(n\in\mathbb{N})$?
\end{enumerate}
The questions (i), (ii) and (iii) are motivated by our observations made in Theorem~\ref{thm:smallerthanoneplusepsilon} (ii) and Theorem~\ref{thm:smallerthanoneplusepsilon2} (v). Concerning (iv) and (v), recall that $0\in\widehat{\mathbb{N}_0}$ implies at least $h(1)\geq2$, see the beginning of Section~\ref{sec:karlinmcgregor}.

\appendix

\section{On the orthogonalization measure of the Karlin--McGregor polynomials}

As announced in Section~\ref{sec:karlinmcgregor}, we sketch a proof of equations \eqref{eq:measureKarlinMcGregorabscont} to \eqref{eq:measureKarlinMcGregordiscr}, which also corrects a small mistake in Karlin and McGregor's paper \cite{KM59}. We first note the following simple observation: let $(P_n(x))_{n\in\mathbb{N}_0}$ be an (arbitrary) orthogonal polynomial sequence as in Section~\ref{sec:intro}, and let $(P_n^{\ast}(x))_{n\in\mathbb{N}_0}\subseteq\mathbb{R}[x]$ be the orthogonal polynomial sequence which corresponds to the probability measure
\begin{equation*}
\mathrm{d}\mu^{\ast}(x):=\frac{1}{a_1}(1-x^2)\,\mathrm{d}\mu(x).
\end{equation*}
Moreover, let $(\sigma_n(x))_{n\in\mathbb{N}_0},(\sigma_n^{\ast}(x))_{n\in\mathbb{N}_0}\subseteq\mathbb{R}[x]$ denote the corresponding monic versions. Then
\begin{equation}\label{eq:bikernelpol}
(1-x^2)\sigma_n^{\ast}(x)=-\sigma_{n+2}(x)+\frac{\sigma_{n+2}(1)}{\sigma_n(1)}\sigma_n(x)
\end{equation}
for all $n\in\mathbb{N}_0$. This is immediate from the relation
\begin{equation*}
\int_{\mathbb{R}}\!(1-x^2)\sigma_n^{\ast}(x)\sigma_k(x)\,\mathrm{d}\mu(x)=a_1\int_{\mathbb{R}}\!\sigma_n^{\ast}(x)\sigma_k(x)\,\mathrm{d}\mu^{\ast}(x)\;(n,k\in\mathbb{N}_0)
\end{equation*}
which, due to orthogonality and symmetry, yields that $(1-x^2)\sigma_n^{\ast}(x)$ has to be a linear combination of $\sigma_{n+2}(x)$ and $\sigma_n(x)$. Concerning the division in \eqref{eq:bikernelpol}, we note that it is clear from Section~\ref{sec:intro} that $\sigma_n(1)>0$ for all $n\in\mathbb{N}_0$.\\

From now on, let $\alpha,\beta\geq2$ and $(P_n(x))_{n\in\mathbb{N}_0}$ be the sequence of Karlin--McGregor polynomials $(K_n^{(\alpha,\beta)}(x))_{n\in\mathbb{N}_0}$. Moreover, let $(\widetilde{P_n(x)})_{n\in\mathbb{N}_0}$ be the sequence given by the modified recurrence relation $\widetilde{P_0(x)}=1$, $\widetilde{P_1(x)}=\beta x/(\beta-1)$,
\begin{equation}\label{eq:modifiedrecurrenceKarlinMcGregor}
x\widetilde{P_n(x)}=a_n\widetilde{P_{n+1}(x)}+c_n\widetilde{P_{n-1}(x)}\;(n\in\mathbb{N}),
\end{equation}
where $(c_n)_{n\in\mathbb{N}}\subseteq(0,1)$ and $a_n\equiv1-c_n$ are the Karlin--McGregor coefficients as in Section~\ref{sec:karlinmcgregor} (so $c_{2n-1}=1/\alpha$, $c_{2n}=1/\beta$). These are the polynomials studied in \cite{KM59}, and up to the aforementioned mistake\footnote{The formula for $\psi^{\prime}(x)$ on \cite[p. 74]{KM59} has to be corrected; the right formula reads $\psi^{\prime}(x)=\sqrt{4p_1q x^2-(x^2-p q_1+p_1q)^2}/(2\pi p_1q|x|)$.} it was shown there that $(\widetilde{P_n(x)})_{n\in\mathbb{N}_0}$ is orthogonal w.r.t. the probability measure $\widetilde{\mu}$ satisfying the following properties: if $\alpha\leq\beta$, then $\widetilde{\mu}$ is absolutely continuous with
\begin{equation}\label{eq:measureKarlinMcGregorabscontmod}
\mathrm{d}\widetilde{\mu}(x)=\begin{cases} \frac{\alpha\beta\sqrt{(\gamma_1^2-x^2)(x^2-\gamma_2^2)}}{2\pi(\alpha-1)|x|}\chi_{(-\gamma_1,-\gamma_2)\cup(\gamma_2,\gamma_1)}(x)\,\mathrm{d}x, & \alpha<\beta, \\ \frac{\alpha^2\sqrt{\gamma_1^2-x^2}}{2\pi(\alpha-1)}\chi_{(-\gamma_1,\gamma_1)}(x)\,\mathrm{d}x, & \alpha=\beta, \end{cases}
\end{equation}
where $\gamma_1$ and $\gamma_2$ are as in Section~\ref{sec:karlinmcgregor}. If $\alpha>\beta$, however, then $\widetilde{\mu}$ consists of the absolutely continuous part
\begin{equation}\label{eq:measureKarlinMcGregorabscontpartmod}
\frac{\alpha\beta\sqrt{(\gamma_1^2-x^2)(x^2-\gamma_2^2)}}{2\pi(\alpha-1)|x|}\chi_{(-\gamma_1,-\gamma_2)\cup(\gamma_2,\gamma_1)}(x)\,\mathrm{d}x
\end{equation}
and the discrete part
\begin{equation}\label{eq:measureKarlinMcGregordiscrmod}
\widetilde{\mu}(\{0\})=\frac{\alpha-\beta}{\alpha-1}.
\end{equation}
It is now left to deduce \eqref{eq:measureKarlinMcGregorabscont} to \eqref{eq:measureKarlinMcGregordiscr} from \eqref{eq:measureKarlinMcGregorabscontmod} to \eqref{eq:measureKarlinMcGregordiscrmod}, which can be done as follows: first, show that the monic versions $(\sigma_n(x))_{n\in\mathbb{N}_0}$ and $(\widetilde{\sigma_n(x)})_{n\in\mathbb{N}_0}$ satisfy the recurrence relations $\sigma_0(x)=\widetilde{\sigma_0(x)}=1$, $\sigma_1(x)=\widetilde{\sigma_1(x)}=x$,
\begin{align*}
x\sigma_n(x)&=\sigma_{n+1}(x)+\lambda_n\sigma_{n-1}(x)\;(n\in\mathbb{N}),\\
x\widetilde{\sigma_n(x)}&=\widetilde{\sigma_{n+1}(x)}+\widetilde{\lambda_n}\widetilde{\sigma_{n-1}(x)}\;(n\in\mathbb{N}),
\end{align*}
where
\begin{equation*}
\lambda_n:=\begin{cases} \frac{1}{\alpha}, & n=1, \\ \frac{\alpha-1}{\alpha\beta}, & n\;\mbox{even}, \\ \frac{\beta-1}{\alpha\beta}, & \mbox{else} \end{cases}
\end{equation*}
and
\begin{equation*}
\widetilde{\lambda_n}=\begin{cases} \frac{\beta-1}{\alpha\beta}, & n=1, \\ \lambda_n, & \mbox{else}. \end{cases}
\end{equation*}
Concerning $(\sigma_n(x))_{n\in\mathbb{N}_0}$, this follows from \eqref{eq:threetermrecmonic}; concerning $(\widetilde{\sigma_n(x)})_{n\in\mathbb{N}_0}$, just calculate the leading coefficients of the polynomials $\widetilde{P_n(x)}$ from \eqref{eq:modifiedrecurrenceKarlinMcGregor} and the initial conditions. Next, use these recurrence relations to show that
\begin{equation*}
(1-x^2)\widetilde{\sigma_0(x)}=-\sigma_2(x)+\frac{\alpha-1}{\alpha}\sigma_0(x)
\end{equation*}
and
\begin{equation*}
(1-x^2)\widetilde{\sigma_n(x)}=-\sigma_{n+2}(x)+\frac{(\alpha-1)(\beta-1)}{\alpha\beta}\sigma_n(x)\;(n\in\mathbb{N})
\end{equation*}
via induction. Finally, use induction to show that
\begin{equation*}
\frac{\sigma_{n+2}(1)}{\sigma_n(1)}=\frac{\sigma_{n+1}(1)}{\sigma_n(1)}-\lambda_{n+1}=\begin{cases} \frac{\alpha-1}{\alpha}, & n=0, \\ \frac{(\alpha-1)(\beta-1)}{\alpha\beta}, & n\in\mathbb{N}. \end{cases}
\end{equation*}
Putting all together, we see that $(\widetilde{\sigma_n(x)})_{n\in\mathbb{N}_0}$ coincides with the sequence $(\sigma_n^{\ast}(x))_{n\in\mathbb{N}_0}$. It is well-known that if the support of an orthogonalization measure is compact, then it is the \textit{unique} orthogonalization measure of its orthogonal polynomial sequence \cite{Ch78}. Therefore, we can conclude that $\widetilde{\mu}=\mu^{\ast}$, which yields the desired formulas for $\mu$.

\bibliography{jointworkrysiekstefan}
\bibliographystyle{amsplain}

\end{document}